\documentclass[11 pt,reqno]{amsart}

\usepackage{amsmath,amsthm,amssymb,amscd,graphicx}
\usepackage{color}
\usepackage{hyperref}
\usepackage{enumerate}
\usepackage{mathrsfs}
\usepackage{enumitem}
\usepackage{multirow}

\newtheorem{thm}{Theorem}[section]
 \newtheorem{cor}[thm]{Corollary}
 \newtheorem{lem}[thm]{Lemma}
 
 \newtheorem{prop}[thm]{Proposition}
 
 \theoremstyle{definition}
 \newtheorem{df}[thm]{Definition}
 \theoremstyle{remark}
 \newtheorem{rem}[thm]{Remark}
 
 \numberwithin{equation}{section}

\def\be#1 {\begin{equation} \label{#1}}
\newcommand{\ee}{\end{equation}}
\renewcommand{\phi}{\varphi}

\def\C{\mathbb C}
\def\R{\mathbb R}
\def\T{\mathbb T}
\def\Z{\mathbb Z}
\def\HH{\mathbb H}
\def\N{\mathbb N}
\def\E{\mathcal E}
\def\W{\mathcal W}
\def\M{\mathcal M}
\def\e{e}
\def\TT{\mathcal T}
\def\pa{\partial}
\def\eps{\epsilon}
\def\dis{\displaystyle}    

\definecolor{gr}{rgb}   {0.,   0.69,   0.23 }
\definecolor{bl}{rgb}   {0.,   0.5,   1. }
\definecolor{mg}{rgb}   {0.85,  0.,    0.85}
\definecolor{yl}{rgb}   {0.8,  0.7,   0.}
\definecolor{or}{rgb}  {0.7,0.2,0.2}

\renewcommand{\Re}{  {\mathfrak{Re}}  }
 
\newcommand{\ov}{  \overline  }

\newcommand\<{\langle}
\renewcommand\>{\rangle}

\begin{document}

\thanks{P. G\'erard is  supported by the grant ``ANAE'' ANR-13-BS01-0010-03. }
\thanks{P. Germain is  supported by the NSF grant DMS-1501019}
\thanks{L. Thomann is supported by the grants ``BEKAM''  ANR-15-CE40-0001,  "ISDEEC'' ANR-16-CE40-0013 and by the ERC Project FAnFAre no. 637510}

\author{Patrick G\'erard}
\address{Laboratoire de Math\'ematiques d'Orsay,
Univ. Paris-Sud, CNRS, Universit\'e Paris--Saclay, 91405 Orsay, France}
\email{Patrick.Gerard@math.u-psud.fr}
\author{Pierre Germain}
\address{Courant Institute of Mathematical Sciences, 251 Mercer Street, New York 10012-1185 NY, USA}
\email{pgermain@cims.nyu.edu}
\author{ Laurent Thomann }
\address{Institut  \'Elie Cartan, Universit\'e de Lorraine, B.P. 70239,
F-54506 Vand\oe uvre-l\`es-Nancy Cedex, FR}
\email{laurent.thomann@univ-lorraine.fr}

\title{On the Cubic Lowest Landau Level Equation}



\begin{abstract}
We study dynamical properties of the  cubic lowest Landau level  equation, which is used in the modeling of fast rotating Bose-Einstein condensates. We  obtain bounds on the decay of general stationary solutions. We then provide a classification of stationary waves with a finite number of zeros.  Finally, we are able to establish which of these stationary waves are stable, through a variational analysis.
\end{abstract}

\maketitle

\tableofcontents

\section{Introduction}

\subsection{The cubic lowest Landau level equation}

Consider, in dimension 2, the magnetic Schr\"odinger operator corresponding to a vertical magnetic field
$$
\Delta_A = \nabla_A \cdot \nabla_A, \qquad \mbox{with} \qquad \nabla_A =  \nabla-iA \quad \mbox{and} \quad  A = \left( \begin{array}{c} -y \\ x \end{array} \right).
$$
From the identity
$$\langle -\Delta _A\psi ,\psi \rangle _{L^2}=2\Vert \psi \Vert _{L^2}^2+\Vert (2\ov{\partial _z}+z)\psi \Vert_{L^2}^2\ , \quad z=x+iy \ ,$$ 
the ground state of $-\Delta_A$  is very degenerate: it consists of the Bargmann-Fock space
$$
\mathcal{E} =\big \{\, u(z) = e^{-\frac{|z|^2}{2}} f(z)\;,\;f \; \mbox{entire\ holomorphic}\,\big\}\cap L^2(\C ),
$$
also called lowest Landau level.

The orthogonal projection on $\mathcal{E}$ is given by the formula (see Paragraph \ref{para12} below) 
\begin{equation}\label{defpi}
[\Pi u](z) = \frac{1}{\pi} e^{-\frac{|z|^2}{2}} \int_\mathbb{C} e^{\ov  w z - \frac{|w|^2}{2}} u(w) \,dL(w),
\end{equation}
where $L$ stands for Lebesgue measure on $\C$.

The cubic lowest Landau level equation is induced by the energy 
$$\displaystyle \mathcal{H}(u) = \frac{1}{4} \int_{\mathbb C} |u|^4 \, dL$$
 given the standard symplectic form $\omega(u,v) =  \mathfrak{Im} \int_{\mathbb{C}} u\overline{v}\,dL$ on $\mathcal{E}$. It reads
\begin{equation}\label{LLL}\tag{LLL}
\left\{
\begin{aligned}
&i\partial_{t}u= \Pi (|u|^2 u), \quad   (t,z)\in \R\times \C,\\
&u(0,z)=  u_0(z).
\end{aligned}
\right.
\end{equation}

\subsection{Derivation of~\eqref{LLL}} This equation arises as a limiting problem in a number of situations.

\subsubsection{Rotating Bose-Einstein condensates} Following~\cite{ABD, Ho}, consider a Bose-Einstein condensate confined by a harmonic field, and rotating at a high velocity. In appropriate coordinates, and for constants $\epsilon$ and $G$, its Hamiltonian reads
$$
\int_\C \Big[ - \big|[\nabla - iA] \psi\big|^2 + \epsilon^2 |z|^2 |\psi|^2 + G |\psi|^4 \Big]\,dL(z)
$$
For $\epsilon, G \ll 1$, the first term is dominant, and, in order to minimize $H$, one can consider that $\psi \in \mathcal{E}$. This leaves us with the Hamiltonian 
$$
\int_\C \big[ \epsilon^2 |z|^2 |\psi|^2 + G |\psi|^4 \big]\,dL(z),
$$
on $\mathcal{E}$. Notice that the corresponding dynamics are the same as that given by $\mathcal{H}$, which is the case $\epsilon = 0$ (we will see later that the term $|z|^2 |\psi|^2$ simply corresponds to rotations, and is therefore harmless).

It is conjectured from physical~\cite{ARVK,BSSD} and numerical~\cite{AB} observations that, as $\epsilon \to 0$, the minimizers (for fixed $L^2$ norm) of the above functional have a very specific structure: within a ball, $u$ is close to a theta function (in particular, its zeros  coincide with an Abrikosov lattice); and away from this ball it decays fast. See~\cite{ABN} for a mathematical approach to this conjecture. 

\subsubsection{Superconductivity} A parallel derivation can be followed for a superconducting material submitted to an exterior magnetic potential: we refer to~\cite{AftaSerfa}.

\subsubsection{Resonant system for a confined nonlinear Schr\"odinger equation} Start this time with the weakly nonlinear Schr\"odinger equation
$$
i \partial_t u - H_0 u = \epsilon^2 |u|^2 u \quad \mbox{with} \quad H_0 = - \Delta + |x|^2.
$$
The completely resonant system, which approximates the evolution of $u$ as $\epsilon \to 0$ is given (after time rescaling) by
\begin{equation}
\label{CR} \tag{CR}
i\partial_t u = \mathcal{T}(u,u,u) \quad \mbox{with} \quad \mathcal{T}(f,f,f) = \int_0^{2\pi} |e^{isH_0} f|^2 e^{isH_0} f\,ds.
\end{equation}
It is derived and studied in~\cite{GHT1,GHT2}. A striking property of~\eqref{CR} is that it agrees with~\eqref{LLL} if its data are chosen in the Bargmann-Fock space $\mathcal{E}$.

The equation~\eqref{CR} can also be derived from the nonlinear Schr\"odinger equation on the torus~\cite{FGH} or in the presence of a magnetic potential~\cite{FMS}. 

\subsection{Comparison with similar equations}
The formulation \eqref{LLL} of the cubic lowest Landau level equation is similar to the cubic Szeg\H{o} equation, introduced by the first author and S. Grellier in \cite{GG1}, and identified in \cite{GG2} as the completely resonant system associated  to the cubic half--wave equation on the circle --- see also Pocovnicu \cite{P1, P2} concerning the cubic Szeg\H{o} equation on the line. An important feature of the cubic Szeg\H{o} equation is that it admits integrability properties through a Lax pair structure satisfied by Hankel operators. Using this structure, traveling wave solutions were classified in \cite{GG1} for the circle and in \cite{P1} for the line, and growth of high Sobolev norms was established in \cite{P2} for the line and  in \cite{GG3} for the circle. 

Though the Lax pair structure for Hankel operators does not seem to extend to \eqref{LLL}, it is therefore natural to study similar questions for equation \eqref{LLL}. A review of our results in this direction is the purpose of the next paragraph. 

Finally, let us mention that the completely resonant system associated to the conformally invariant cubic wave equation on the three--dimensional sphere was recently introduced in \cite{BCEHLM}. 

\subsection{Main results}
In this paragraph, we briefly describe the main results of this paper. We recall that
$$
\mathcal{E} =\big \{\, u(z) = e^{-\frac{|z|^2}{2}} f(z)\;,\;f \; \mbox{entire\ holomorphic}\,\big\}\cap L^2(\C ).
$$
\subsubsection{The initial value problem and long time Sobolev bounds}

The well-posedness of \eqref{LLL} was studied by F. Nier~\cite[Proposition 3.1]{Nier} (see Remark~\ref{eqNier}), and the following statement holds true.
\begin{thm}\label{mainCauchy}
For every $u_0\in \mathcal{E}$, there exists a unique solution $u\in \mathcal{C}^\infty (\R ,\mathcal{E})$ to equation \eqref{LLL}, and this solution depends smoothly on $u_0$. Moreover, for every  $t\in \R$ 
$$\int_\C |u(t,z)|^2\, dL(z)=\int_\C |u_0(z)|^2\, dL(z).$$
 Furthermore, if moreover $zu_0\in L^2(\C )$, then  $zu(t)\in L^2(\C)$ for every $t\in \R$, and 
$$\int_\C |z|^2|u(t,z)|^2\, dL(z)=\int_\C |z|^2|u_0(z)|^2\, dL(z)\ ,\ \int_\C z|u(t,z)|^2\, dL(z)=\int_\C z|u_0(z)|^2\, dL(z)\ .$$
More generally, if, for some $s>0$, $\langle z\rangle ^s u_0\in L^2(\C)$, then  $\langle z\rangle ^su(t)\in L^2(\C)$ for every $t\in \R$.
\end{thm}

Our next result concerns the long time bounds for Sobolev norms, which are equivalent to weighted norms $\Vert \langle z\rangle ^ku\Vert _{L^2}$ --- see Lemma \ref{lemEq} below.
\begin{thm}\label{bounds}
With the notation of Theorem \ref{mainCauchy}, assume $\langle z\rangle ^k u_0\in L^2(\C)$ for some integer $k\geq 1$. Then
$$\Vert \langle z\rangle ^ku(t)\Vert _{L^2(\C)}\leq C_k(1+|t|)^{\frac{k-1}{2}}\ .$$
\end{thm}
Notice that Theorem \ref{bounds} is in strong contrast with the results of \cite{GG3} for the cubic Szeg\H{o} equation on the circle, where superpolynomial growth of Sobolev norms is established to be generic in the Baire sense.
On the other hand, we mention in this paper two results improving the above polynomial rate for a perturbation of \eqref{LLL} under generic Hermite multipliers. Theorem \ref{thFN} is a direct consequence of normal form theory for semilinear quantum harmonic oscillators \cite{GIP} and states that, for any exponent $r$ and for a full measure set of Hermite multipliers of any given algebraic decay, solutions having an initial data of order $\epsilon $ in a big Sobolev space conserve the same size on a time of length $\epsilon ^{-r}$. Theorem \ref{KAMLLL} is a direct application of KAM theory \cite{GT} to this context and allows to find small quasiperiodic solutions --- hence uniformly small in any Sobolev space ---  for the perturbation of \eqref{LLL} by a subset of Hermite multipliers of asymptotically full measure.

\subsubsection{Stationary waves}
In view of the two dimensional invariance by phase rotations and space rotations, it is natural to define stationary waves for equation \eqref{LLL} as solutions of the form
$$u(t,z)=e^{-i\lambda t}u_0\left (e^{-i\mu t}\right )\, ,$$
for some $(\lambda ,\mu )\in \R^2$. Equivalently, the corresponding initial condition $u_0$, also called a stationary wave, satisfies 
\begin{equation}\label{eq:statwave}
\lambda u_0+\mu \Lambda u=\Pi (|u|^2u)\  ,\ \Lambda :=z\pa _z-\overline z\overline {\pa _z} .
\end{equation}
We obtain several results about these special solutions. Firstly, we provide a priori bounds on the growth at infinity of any stationary wave.
\begin{thm}\label{boundstatwave}
Let $u_0\in \mathcal E$ be a solution of \eqref{eq:statwave}. Then, for any
$$
\eta>\eta_0 = \left( \frac{1}{2} + \frac{1}{2} \frac{\log 2}{\log 3} \right)^{-1} \sim 1.226\dots,
$$
the following estimate holds,
 $$\forall z\in \C\ ,\ |u_0(z)| \leq C_\eta e^{|z|^\eta-\frac{1}{2}|z|^2}\ .$$ 
 As a consequence, if
 $$
N(R) = \# \big\{ z \in \mathbb{C} \; \mbox{such that} \; u(z) = 0 \; \mbox{and} \; |z|<R \big\} ,
$$
then for any $\eta > \eta_0$, 
$$
\frac{N(R)}{R^\eta} \longrightarrow 0 \quad \mbox{ as } \;\;R \to \infty.
$$
\end{thm}
Secondly, we classify stationary waves with a finite number of zeros and we study their orbital stability in $\mathcal E$ and in
$$L^{2,1}_\mathcal E:=\{ u\in \mathcal E : zu\in L^2(\C)\}\ .$$
\begin{thm}\label{statwavefinite}
Up to multiplicative factors, phase rotations and space rotations, the stationary waves in $\mathcal E$ having only a finite number of zeros  are
$$
u_n^\alpha(z) = (z-\overline \alpha)^n e^{-\frac{|z|^2}{2}-\frac{|\alpha|^2}{2} + \alpha z}, \quad \alpha \in\C
$$
for which $\mu =0$, and 
$$
v_b(z) =\left( z - \frac{b(2+b^2)}{1+b^2}  \right) e^{-\frac{1}{2}|z|^2 + \frac{b}{1+b^2} z}, \quad  0\leq b <\infty.
$$
Furthermore, $u_0^\alpha $ and $u_1^\alpha $ are orbitally stable in $L^{2,1}_\mathcal E$ for phase rotations, $v _b$ is orbitally stable in $L^{2,1}_\mathcal E$ for phase and space rotations, and $u_n^\alpha , n\ge 2,$ are not orbitally stable. Finally, the set
$$\{ e^{i\theta}u_0^\alpha , \,\theta \in \T ,\, \alpha \in \C\}$$
is stable in $\mathcal E$.
\end{thm}
Our third class of results about stationary waves concern existence of stationary waves with an infinite number of zeros. 
We construct these objects using three different methods. Firstly, by bifurcation from $u_n^0 $ --- see Proposition~\ref{bifurc}. Secondly, by a minimization argument combined with the classification result of stationary waves having only a finite number of zeros--- see Proposition~\ref{PropminP}. Finally, by explicit formulae we construct stationary waves having zeros on sets $\gamma {\mathbb Z}$ and $\gamma {\mathbb Z}\cup \frac{i\pi}{k\ov{\gamma}}{\mathbb Z}$, where $\gamma \ne 0$ is an arbitrary complex number, and $k\ne 0$ is an arbitrary integer--- see Appendix \ref{explicitstatwave}. In the first case, the growth at infinity is at most $e^{c|z|-|z|^2/2}$, while this growth is optimal  in the third case. We have not been able to find stationary waves with a faster growth at infinity.

\subsubsection{Number of zeros of the minimizer} We now turn to the question of the number of zeros (in particular, finite or not) of minimizers of a physically relevant variational problem involving the conserved quantities of the equation. In order to describe the results obtained in this respect, we switch to semi-classical coordinates, which are most commonly used in this context.

Let $0<h<1$ be a small parameter and denote by
 $$\mathcal{E}_h =\big \{\, v(w) = e^{-\frac{|w|^2}{2h}} g(w)\;,\;g \; \mbox{entire\ holomorphic}\,\big\}\cap L^2(\C ).$$
Define the energy functional 
\begin{equation}\label{nrj}
E^{h}_{LLL}(v)=\int_{\C}\Big(|w|^2|v(w)|^2+\frac{Na\Omega^2_h}2|v(w)|^4\Big)dL(w),
\end{equation}
where $N,a>0$ are parameters, and $\Omega_h^2=1-h^2$. Consider  the minimizing problem
\begin{equation}\label{mini}
\min_{\substack{v \in \mathcal{E}_h \\ M(v) =1}} E^{h}_{LLL}(v).
\end{equation}
In \cite[Theorem 1.2]{ABN},  Aftalion, Blanc and Nier give conditions on $0<h<1$ and on the Lagrange multiplier associated to the problem \eqref{mini} such that the global minimizer of \eqref{nrj} at fixed mass has an infinite number of zeros. Thanks to the classification result of Theorem~\ref{statwavefinite} combined with a global analysis, we are able to weaken their conditions. Moreover, we prove that the  Gaussian is the unique global minimizer for  an explicit range of the parameter $h>0$. Our result reads

\begin{thm}\label{thmmini} Set $\kappa_0=\frac5{32}$ and $\kappa_1=\sqrt{3}-1$. 

$(i)$ Assume that 
\begin{equation}\label{56}
h<\sqrt{\kappa_0\frac{Na \Omega^2_h}{4 \pi}}.
\end{equation}
Then every local or global minimizer of \eqref{mini} has an infinite number of zeros.

$(ii)$ Assume that 
\begin{equation}\label{57}
h>\sqrt{\kappa_1\frac{ Na \Omega^2_h}{4 \pi}}.
\end{equation}
Then 
$$\phi_{0,h}(z)= \frac{1}{\sqrt{\pi h }}  e^{-\frac{|z|^2}{2h}}$$
is the unique global minimizer of \eqref{mini} and
$$E^{h}_{LLL}(\phi_{0,h})=\frac{Na \Omega_h^2}{4\pi h}+h.$$
\end{thm}
As mentioned earlier, a question of great interest is the localization of the zeros of the minimizer in the case \eqref{56}. The result of Theorem~\ref{boundstatwave} gives some information about the distribution of the zeros, and we refer  to \cite{AB}  for a study of minimizing sequences whose zeros are localized on lattices.

\subsection{Organization of the paper} Section~\ref{Sect2} is devoted to general background about the equation~\eqref{LLL}. Local and global well-posedness results are established in Section~\ref{Sect3}. In Section~\ref{Sect4}, we prove Theorem~\ref{bounds}, and two results improving this polynomial bound for  perturbations of~\eqref{LLL} under generic Hermite multipliers. The last three sections are devoted to stationary waves for~\eqref{LLL}. In Section~\ref{Sect5}, we prove general a priori bounds on the growth of stationary waves at infinity. In Section~\ref{Sect6}, we classify stationary waves with a finite number of zeros, and show how to construct others by perturbation from some of them. Section~\ref{Sect7} deals with various variational problems leading to stationary waves, with applications to stability theory of stationary with a finite number of zeros.  Finally, three appendices are devoted to some explicit stationary waves, a dictionary through Bargmann transform, and an elementary characterization of Sobolev spaces at the lowest Landau level.

\section{Symmetries, conserved quantities and special coordinates}\label{Sect2}
 In this whole section, we present the structure of~\eqref{LLL} while remaining at a formal level.  

\subsection{Hamiltonian structure}
Define first the symplectic form on $\mathcal{E}$:
$$
\omega(u,v) =  \mathfrak{Im} \int_{\mathbb{C}} u\overline{v}\,dL.
$$
Given a functional $F$ on $\mathcal{E}$, its symplectic gradient $\nabla_\omega F \in \mathcal{E}$ is such that $\omega( \nabla_\omega F(u)\,,\, \varphi) = dF(u)\cdot \varphi$. The Hamiltonian flow associated to $F$, denoted $\varphi_F$, is then defined for $t \in \mathbb{R}$ by
$$
\varphi_F(t) u_{0} = u(t) \quad \mbox{where $u$ solves} \quad \left\{ \begin{array}{l} \partial_t u(t) = - \nabla_\omega F(u) \\[3pt] u(t=0) = u_0. \end{array} \right.
$$
The equation~\eqref{LLL} corresponds to the Hamiltonian flow for the Hamiltonian 
$$
\mathcal{H}(u) =\frac 14  \int_{\mathbb{C}} |u|^4 \,dL \qquad \mbox{on $\mathcal{E}$}.
$$
In other words, the solution of~\eqref{LLL} with data equal to $u_0$ can be written $u(t) = \varphi_\mathcal{H}(t) u_0$.

Observe that the Hamiltonian $\mathcal{H}$ is left invariant by the following symmetries: phase rotations
\begin{equation*} 
T_{\gamma} : u(z) \mapsto e^{i \gamma }u(z)  \qquad \mbox{for $\gamma \in \mathbb{T}$},
\end{equation*}
 space rotations 
\begin{equation*} 
L_{\phi} : u(z) \mapsto u(e^{i\phi}z)  \qquad \mbox{for $\phi \in \mathbb{T}$},
\end{equation*}
and magnetic translations 
\begin{equation}\label{defR}
R_{\alpha} : u (z) \mapsto u(z+\alpha) e^{\frac{1}{2}(\overline z \alpha - z \overline{\alpha})} \qquad \mbox{for $\alpha \in \mathbb{C}$}.
\end{equation}

These symmetries are via Noether theorem related to quantities which are  invariant by the flow of~\eqref{LLL}: the mass $M$, angular momentum $P$, and magnetic momentum $Q$ which are given, for $u \in \mathcal{E}$, by
\begin{align*}
&M(u) = \int |u|^2 dL \\
&P(u) = \int  \Lambda u \, \overline{u}\,dL= \int_{\mathbb{C}} (\vert z\vert^2-1)\vert u(z)\vert^2 \,dL(z) \\
&Q(u) = Q_x(u) + i Q_y(u)= \int_{\C} z |u|^2(z) dL(z) ,
\end{align*}
where $\Lambda$ is the angular momentum operator, defined by
$$
\Lambda = i (y\partial_x - x \partial_y) = z\partial_z - \overline{z} \overline{\partial_z},
$$
with $\partial_z=\frac12(\partial_x-i\partial_y)$. The harmonic oscillator $H$ is defined by
$$
H = -4\partial_z \partial_{\ov z}+|z|^2.
$$
It is clear that $M$, $P$ and $Q$ are left invariant by phase and space rotation; magnetic rotations~$R_\alpha$ leave $M$ invariant but act on $Q$ and $P$ as follows, 
\begin{equation*} 
Q(R_\alpha u)=Q(u)-\alpha M(u)\ ,
\end{equation*} 
and
\begin{equation*} 
P(R_\alpha u) = P(u) - 2 \mathfrak{Re} (\overline \alpha Q(u)) + |\alpha|^2 M(u).
\end{equation*}

The following table recapitulates for each quantity conserved by the flow of~\eqref{LLL} the corresponding symplectic gradient and the generated symmetry.\medskip

\begin{center}
\begin{tabular}{ | c |c | c |}
\hline
Conserved quantity & Symplectic gradient & Symmetry \\ \hline
\quad Mass \quad & & \\
$M(u) = \int |u|^2 \,dL$ &  $2iu(z)$ & $ T_{\gamma} u(z)  = e^{i \gamma} u(z), \; \gamma \in\R$ \\ \hline
\quad Angular momentum \quad & & \\
$P(u) = \int \Lambda u \overline{u} \,dL$ & $2i\Lambda u(z)$ & $L_\phi u(z) = u(e^{i \phi } z), \; \phi\in\R$ \\
\hline
\quad Magnetic momentum \quad& & \\
$Q_x(u) = \int x|u|^2 \,dL$ & $\quad 2i\Pi(xu) = i \Big( z + \partial_z + \frac{\ov  z}{2} \Big) u(z)\quad $&  $R_{i\beta} u = u(z+i \beta) e^{i\beta x}, \; \beta \in\R$\\ 
$Q_y(u) = \int y|u|^2 \,dL$ & $\quad 2i\Pi(yu) = \Big( z - \partial_z - \frac{\ov  z}{2} \Big) u(z)$ \quad& $R_\alpha u = u(z+\alpha) e^{-i \alpha y}, \; \alpha\in\R$ \\ \hline
\end{tabular}
\end{center}
\medskip

Notice that the phase rotation $T_\gamma$ obviously commutes with all the other symmetries, but this is not the case for $L_\phi$, $R_{\alpha}$ and $R_{i \beta}$, for $\gamma, \phi, \alpha, \beta \in \R$.

\begin{rem}\label{eqNier}
The equation $i\partial_{t}v-\Lambda v= \Pi (|v|^2 v)$ which derives from the Hamiltonian  $\mathcal{H}(u) + 2P(u)$ was studied by Nier\;\cite{Nier}. Since the Hamiltonian flows generated by $\mathcal{H}$ and $P$ commute, it is equivalent to~\eqref{LLL}.
\end{rem}
\medskip

\subsection{The basis $(\varphi_n)$}\label{para12} Denote by $(\phi_n)_{n \geq 0}$ the family of the special Hermite functions given by 
$$
\varphi_n(z) = \frac{1}{\sqrt{\pi n!}} z^n e^{-\frac{|z|^2}{2}}.
$$
By~\cite[Proposition 2.1]{Zhu}, the family $(\phi_n)_{n \geq 0}$ forms  a Hilbertian basis of $\mathcal{E}$, and we can check that they are the eigenfunctions in $\mathcal{E}$ of $H$, $\Lambda$ and of the Fourier transform\footnote{with the normalization $
\mathcal{F} f(\xi) = \frac{1}{2\pi} \int_\C e^{-i\xi \cdot z} f(z) dL(z)$, where $\xi .z:= \Re(\xi \overline z) $. } $\mathcal{F}$ 
\begin{equation*}
H \varphi_n = 2(n+1) \varphi_n, \quad\Lambda \varphi_n = n \varphi_n, \quad \mathcal{F} \varphi_n = i^n \varphi_n.
\end{equation*}
Observe that the Fourier transform will not play any particular role, since $\mathcal{F}=L_{\frac{\pi}2}$. Incidentally, this implies the invariance of the equation under $\mathcal{F}$.\medskip

The   kernel of the projector $\Pi$ is given by
\begin{equation*}
K(z,\xi)=\sum_{n=0}^{+\infty}\phi_n(z)\ov{\phi_n}(\xi)=\frac{1}{\pi}e^{\ov{\xi}z}e^{-\vert \xi\vert^2/2}e^{-\vert z\vert^2/2}, \quad   (z,\xi)\in \C\times \C,
\end{equation*} 
which leads to the formula \eqref{defpi}. \medskip

Decomposing $u$ in this basis
$$
u = \sum_{n=0}^{+\infty} c_n \varphi_n,
$$
the conserved quantities become
$$ \mathcal{H}(u) =  \frac{1}{8\pi}\sum_{\substack{ k,\ell,m,n \geq 0 \\ k + \ell = m + n }} \frac{(k + \ell)!}{2^{k+\ell} \sqrt{k! \ell ! m ! n!}} \overline{c_k} \overline{c_\ell} c_m c_n =   \frac{1}{8\pi }\sum _{\ell =0}^{+\infty} \frac 1{2^{\ell }}\left \vert \sum _{n+p=\ell} c_nc_p\left (\frac{(n+p)!}{n!p!}\right )^{1/2}  \right \vert ^2 $$
\begin{align*}
& M(u) = \sum_{n=0}^{+\infty} |c_n|^2 \\
& P(u) = \sum_{n = 1}^{+\infty} n |c_n|^2 \\
& Q(u) = \sum_{n=0}^{+\infty} \sqrt{n+1} c_n \overline{c_{n+1}},
\end{align*}
see~\eqref{sst} and \cite{GHT1}, while~\eqref{LLL} reads
\begin{equation}
\label{LLLc}
i \partial_t c_k =  \sum_{\substack{ \ell,m,n \geq 0 \\ k + \ell = m + n }} \frac{1}{2\pi} \frac{(k + \ell)!}{2^{k+\ell} \sqrt{k! \ell ! m ! n!}}  \overline{c_\ell} c_m c_n, \quad k\geq 0.
\end{equation}

\subsection{Tempered distributions}
Sometimes we will need to work in the following enlarged lowest Landau level space,
\begin{equation*} 
\tilde {\mathcal E}:=\left \{ u(z) = e^{-\frac{|z|^2}{2}} f(z)\;,\;f \; \mbox{entire\ holomorphic}\, \right \}\cap \mathscr{S}'({\mathbb C})=\left \{ u\in \mathscr{S}'(\C ), \ov {\pa_z}u+\frac	{z}{2}u=0\right \} \ .
\end{equation*}
One can easily establish  that elements of $\tilde{\mathcal E}$ are series of the form
$$u=\sum_{n=0}^{+\infty} c_n \varphi _n\ ,$$
where the sequence $(c_n)$ has at most a polynomial growth in $n$.

Observe that $H=2(\Lambda+1)$  on $\widetilde {\mathcal E}$ . \medskip

\section{Well-posedness}\label{Sect3}

\subsection{Local well-posedness in $z$ coordinates}

For $p \in [1,\infty]$, the weighted $L^p$ space $L^{p,\alpha}$ is given by the norm
$$
\| f \|_{L^{p,\alpha}} = \| \langle z \rangle^\alpha f(z) \|_{L^p(\mathbb{C})}.
$$
Define then
\begin{align*}
& L^p_\E =\big \{\, u(z) = e^{-\frac{|z|^2}{2}} f(z)\;,\;f \; \mbox{entire\ holomorphic}\,\big\}\cap L^p(\C )\\
& L^{p,\alpha}_\E = \big \{\, u(z) = e^{-\frac{|z|^2}{2}} f(z)\;,\;f \; \mbox{entire\ holomorphic}\,\big\}\cap L^{p,\alpha}(\C ).
\end{align*}
These are Banach spaces when endowed with their natural norms.
A classical estimate~\cite{Zhu} gives the embedding of $L^p_\mathcal{E}$ in $L^q_{\mathcal{E}}$ for $p < q$; the inequality with the optimal constant reads~\cite{Carlen}, for all~$u \in \mathcal{E}$
\begin{equation} 
\label{hypercontractivity}
\mbox{if $1 \leq p \leq q \leq \infty$,} \qquad\left( \frac{q}{2\pi} \right)^{1/q} \| u \|_{L^q(\C)} \leq \left( \frac{p}{2\pi} \right)^{1/p} \| u \|_{L^p(\C)}.
\end{equation}

In order to discuss~\eqref{LLL} in $L^p_\mathcal{E}$, we need to extend $\Pi$ to $L^p$; this is easily achieved.

\begin{prop}\label{prop22}
For any $p\in [ 1,\infty]$ and $\alpha \geq 0$, the projector $\Pi$ has a unique bounded extension to $L^p$ and $L^{p,\alpha}$, which is given by the kernel $\frac{1}{\pi} e^{-\frac{|z|^2}{2} - \frac{|w|^2}{2} + \overline{w}z}$.
\end{prop}
\begin{proof}
The kernel $K(z,w) = \frac{1}{\pi} e^{-\frac{|z|^2}{2} - \frac{|w|^2}{2} + \overline{w}z}$ enjoys Gaussian bounds: $|K(z,w)| \leq \dis  \frac1{\pi}  e^{-\frac{|z-w|^2}{2}}$. Therefore, for $u \in L^2 \cap L^{p,\alpha}$,
\begin{align*}
\| \Pi u \|_{L^{p,\alpha}} & = \left\| \langle z \rangle^\alpha \int_\C K(z,w) u(w) \,dL(w) \right\|_{L^p} \lesssim  \left\|  \int_\C e^{-\frac{1}{2}|z-w|^2} \big[\langle w \rangle^\alpha + \langle z- w \rangle^\alpha \big] u(w) \,dL(w) \right\|_{L^p} \\
& \lesssim \| \langle z \rangle^\alpha u(z)\|_{L^p} = \| u \|_{L^{p,\alpha}}.
\end{align*}
Since $L^2$ is dense (in the weak sense for $p=\infty$) in $L^{p,\alpha}$, we obtain a unique bounded extension of the projection operator $\Pi$.
\end{proof}

With this extension of $\Pi$ to $L^{p,\alpha}$ for any $p$, the meaning of~\eqref{LLL} for $u \in L^\infty([0,T],L^{p,\alpha})$ is now clear.

\begin{prop}\label{wpLp}
\begin{itemize}
\item[(i)] ($L^p$ spaces) For any $p \in [ 1,\infty]$, the equation~\eqref{LLL} is locally well-posed in $L^p$: for any data $u_0$ in $L^p_\mathcal{E}$, there exists $T>0$ and a unique solution in $L^\infty([0,T],L^p_\mathcal{E})$, which depends smoothly on $u_0$.
\item[(ii)] (Weighted $L^p$ spaces) For any $p \in [1,\infty]$, $\alpha \geq 0$, the equation~\eqref{LLL} is locally well-posed in $L^{p,\alpha}$: for any data $u_0$ in $L^{p,\alpha}_\mathcal{E}$, there exists $T>0$ and a unique solution in $L^\infty([0,T],L^p_\mathcal{E})$, which depends smoothly on $u_0$.
\end{itemize}
\end{prop}

\begin{proof} Local well-posedness is obtained from the theory of ordinary differential equations, by observing that the vector field 
$$u\mapsto \Pi (|u|^2u)$$
is smooth on the spaces $L^{p,\alpha }$, $1\leq p\leq \infty , \alpha \geq 0$, with a differential bounded on bounded subsets. This observation uses  successively the boundedness of $\Pi$,  and the $L^p$-$L^q$ estimate \eqref{hypercontractivity},
$$
\| \langle z \rangle^\alpha \Pi \big( a\overline b c\big) \|_{L^p} \lesssim \| \langle z \rangle^\alpha a\overline bc \|_{L^p}= \| \langle z \rangle^\alpha a\|_{L^{p}} \|  b \|_{L^{\infty}}\| c\| _{L^\infty} \lesssim \|\langle z \rangle^\alpha a\|_{L^p} \|\langle z \rangle^\alpha b\|_{L^p}\|\langle z \rangle^\alpha c\|_{L^p}.
$$
\end{proof}

\begin{rem} The space $L^\infty$ is the endpoint space as far as local well-posedness is concerned. Smaller data spaces, such as $L^p$, with $p<\infty$, or $L^{\infty,\alpha}$, with $\alpha>0$, enjoy stronger properties:
\begin{itemize}
\item Smoothing effect: if $u_0\in L^p_\mathcal{E}$, then for any $ \in [0,T]$, $u(t) - u_0 \in L^{\max(1,\frac{p}{3})}_\mathcal{E} \cap L^\infty_\mathcal{E}$.
\item Weak compactness: if $(u_k)$ is a sequence of solutions uniformly bounded in $L^\infty([0,T],L^p)$, there exists a solution $u \in L^\infty([0,T],L^p)$ such that $u_k(t)$ converges weakly in $L^p$ to $u(t)$.
\end{itemize}
The proofs are immediate and we omit them.
\end{rem}

\subsection{Local well-posedness in $(c_k)$ coordinates} Let $\alpha\geq 0$ and $\lambda>0$. Denote $\ell^{\infty,\alpha}$ and $C_\lambda$ the Banach spaces of sequences given by the norms
\begin{equation}\label{defC}
\| (c_k) \|_{\ell^{\infty,\alpha}} = \sup_{k \geq 0} \langle k \rangle^\alpha |c_k| \quad \mbox{and} \quad \| (c_k) \|_{C_\lambda} = \sup_{k \geq 0} \frac{\sqrt{k !}}{\lambda^k} |c_k|.
\end{equation}

\begin{prop} \label{propC}
Using the coordinates $(c_k)$ given by $\dis u = \sum_{k=0}^{+\infty} c_k \phi_k$, the equation~\eqref{LLL} is locally well posed 
\begin{itemize}
\item[(i)] in $\ell^{\infty,\alpha}$ for $\alpha \geq \frac{1}{4}$.
\item[(ii)] in $C_\lambda$ for $\lambda > 0 $.
\end{itemize}
\end{prop}

\begin{rem} The spaces $\ell^{\infty, 1/4}$ and $C_\lambda$ are of particular relevance, as will become clear in the remainder in this article. Roughly speaking, they are, in $(c_n)$ coordinates, the largest and the smallest space for which local well-posedness holds.
\end{rem}

\begin{proof} $(i)$ Recall that the equation~\eqref{LLL} written in $(c_n)$ coordinates reads
\begin{align*}
i \partial_t c_k & = \frac{1}{2 \pi}\sum_{\substack{ \ell,m,n \geq 0 \\ k + \ell = m + n }} \frac{(k+\ell)!}{2^{k+\ell} \sqrt{k! \ell ! m! n!}} \overline{c_\ell} c_m c_n \\
& = \frac{1}{2 \pi} \sum_{S=k}^\infty \sum_{m=0}^S \sqrt{\frac{S!}{2^S k! (S-k)!}}\sqrt{\frac{S!}{2^S m! (S-m)!}}\overline{c_{S-k}} c_m c_{S-m} \\
& =: \TT(c,c,c).
\end{align*}
We need some bounds on the interaction coefficients: Stirling's formula gives the inequality
$$
\sqrt{\frac{S!}{2^S k! (S-k)!}} \lesssim \psi\left( \frac{k}{S} \right)^S  \frac{\langle S \rangle^{1/4}}{\langle k \rangle^{1/4} \langle S - k \rangle^{1/4}},
$$
where we denote, if $0<x<1$, $\displaystyle \psi(x) = \sqrt{\frac{1}{2 x^x (1-x)^{1-x}}}$. One checks that $\psi(x)$ takes values in $(0,1$), is  equal to $1$ only if $x=\frac{1}{2}$, and satisfies the bound $|\psi(x)| \leq e^{-C(x-\frac{1}{2})^2}$.

In order to prove the proposition, it suffices to show that $\TT$ maps $(\ell^{\infty,\alpha})^3 \to \ell^{\infty,\alpha}$, which would follow from the inequality 
$$
\Sigma(k) \lesssim \frac{1}{\langle k \rangle^{\alpha}},
$$
where
$$
\Sigma(k) = \sum_{S=k}^{+\infty} \sum_{m=0}^S  \psi\left( \frac{k}{S} \right)^S  \psi\left( \frac{m}{S} \right)^S \frac{\langle S \rangle^{1/2}}{\langle k \rangle^{1/4} \langle S - k \rangle^{1/4}\langle m \rangle^{1/4} \langle S - m \rangle^{1/4}} \frac{1}{\langle S - k \rangle^{\alpha}\langle m \rangle^{\alpha} \langle S - m \rangle^{\alpha}}.
$$
In order to prove that this inequality holds, we first consider the sum over $m$:
$$
\sum_{m=0}^S  \psi\left( \frac{m}{S} \right)^S \frac{1}{\langle m \rangle^{\frac{1}{4}+\alpha} \langle S-m \rangle^{\frac{1}{4}+\alpha}} \lesssim \sum_{m=0}^S  e^{-CS(\frac{m}{S} - \frac{1}{2})^2} \frac{1}{\langle m \rangle^{\frac{1}{4}+\alpha} \langle S-m \rangle^{\frac{1}{4}+\alpha}} \lesssim \frac{1}{\langle S \rangle^{2\alpha}}.
$$
It remains to sum over $S$:
$$
\Sigma(k) \lesssim \sum_{S=k}^{+\infty}  e^{-CS(\frac{k}{S} - \frac{1}{2})^2} \frac{\langle S \rangle^{\frac{1}{2}-2\alpha} }{\langle k \rangle^{1/4} \langle S-k \rangle^{\frac{1}{4}+\alpha}} \lesssim \langle k \rangle^{\frac{1}{2}-3\alpha} \lesssim \langle k \rangle^{-\alpha},
$$
where the last inequality follows from $\alpha \geq \frac{1}{4}$.

\bigskip

\noindent $(ii)$ Proceeding as in the previous point, it suffices to show that
$$
\sup_{k \geq 0} \frac{\sqrt{k!}}{\lambda^k}\sum_{\substack{ \ell,m,n \geq 0 \\ k + \ell = m + n }} \frac{(k+\ell)!}{2^{k+\ell} \sqrt{k! \ell ! m! n!}} \frac{\lambda^{m+n+\ell}}{\sqrt{\ell!} \sqrt{m!} \sqrt{n!} }= \sup_{k \geq 0} \sum_{k+\ell = m + n} \frac{\lambda^{2 \ell}(k+\ell)!}{2^{k+\ell}  \ell ! m! n!} < \infty.
$$
Setting $p = k +\ell$, this can also be written, using the binomial identity,
\begin{multline*}
\sup_{k \geq 0} \sum_{p \geq k} \sum_{n \leq p}  \frac{\lambda^{2 (p-k)}p!}{2^{p}  (p-n)! n! (p-k)!} 
=\sup_{k \geq 0} \sum_{p \geq k} \frac{\lambda^{2 (p-k)}}{(p-k)!}  \sum_{n \leq p}  \frac{p!}{2^{p}  (p-n)! n!} 
=  \sup_{k \geq 0} \sum_{p \geq k} \frac{\lambda^{2 (p-k)}}{(p-k)!} =e^{\lambda^2},
\end{multline*}hence the result.
\end{proof}

The following lemma shows how the critical spaces in $z$ space ($L^\infty$) and $(c_n)$ space ($\ell^{\infty,1/4}$) are related.

\begin{lem} $(c_n) \in \ell^{\infty,1/4}$,
\begin{equation}\label{ineq24}
\left\| \sum_{n=0}^{+\infty} c_n \varphi_n \right\|_{L^\infty(\C)} \lesssim \| (c_n) \|_{\ell^{\infty,1/4}}.
\end{equation}
\end{lem}

\begin{proof} First observe that
$$ e^{-\frac{1}{2}|z|^2 + \frac{1}{2} z^2} = \sum_{n=0}^{+\infty} \frac{\sqrt{\pi (2n)!}}{2^n n!} \phi_{2n}(z),
$$
which implies since $\frac{\sqrt{\pi (2n)!}}{2^n n!} \sim \frac{(2\pi)^{1/4}}{2 n^{1/4}}$ that
\begin{equation*}
 \sup_{z\in \C} \,\sum_{n=0}^{+\infty} \frac{ \phi_{2n}(|z|)}{(n+1)^{1/4}} <\infty.
\end{equation*}
Using this inequality and $ |\phi_{2n+1}| \leq  |\phi_{2n}|+ |\phi_{2(n+1)}|$, this gives for $\dis u=\sum_{n=0}^{+\infty} c_n \varphi_n$ with $(c_n) \in \ell^{\infty,1/4}$ that 
\begin{equation*}
 \sup_{z\in \C} \,|u(z)| \lesssim  \sup_{z\in \C} \, \sum_{n=0}^{+\infty}\frac{\phi_{n}(|z|)}{ (n+1)^{1/4}}\lesssim  \sup_{z\in \C} \,\sum_{n=0}^{+\infty}\frac{\phi_{2n}(|z|)}{ (n+1)^{1/4}}+  \sup_{z\in \C} \,\sum_{n=0}^{+\infty}\frac{\phi_{2n+1}(|z|)}{ (n+1)^{1/4}} < \infty.
\end{equation*}
\end{proof}
Notice that the reverse inequality in \eqref{ineq24} does not hold true, as can be seen by considering  the sequence $u_n=n^{-1/4} \phi_n$.

\subsection{Global well-posedness} The conservation of $M$ and $\mathcal{H}$ combined with the local well-posedness in $L^2_{\mathcal{E}}$ and $L^4_{\mathcal{E}}$ easily leads to

\begin{prop}
Assume that $2\leq p\leq 4$. The equation~\eqref{LLL} is globally well-posed for data in~$L_{\mathcal E}^p$ and such data lead to solutions in $\mathcal{C}^\infty \big(\R,L_{\mathcal E}^p\big)$, depending smoothly on the initial data. 

Moreover, for $u_0 \in L^p_{\mathcal{E}}$, 
\begin{equation}\label{bound1}
 \|u(t)-u_0\|_{L^p(\C)} \lesssim |t|^{4/p-1},\qquad  \|u(t)-u_0\|_{L^2(\C)} \leq C|t| , \qquad \forall t\in \R.
\end{equation}
\end{prop}

\begin{proof}
We already know local well-posedness from Proposition \ref{wpLp}. Furthermore, using successively the boundedness of $\Pi$ (Proposition~\ref{prop22}), H\"older's inequality, and\;\eqref{hypercontractivity},
$$
\big\| \Pi (|u|^2 u) \big\|_{L^p} \leq C_1 \big\| |u|^2 u \big\|_{L^p}= C_1 \| u \|_{L^{3p}}^3 \leq  C_2 \|u\|_{L^4}^2\|u\|_{L^p}.
$$
The previous inequality shows that the lifespan of the solution only depends on the $L^4$ norm which is preserved, hence we get global  well-posedness.

Let us now prove the bound \eqref{bound1}.  We write $u=u_0+v$, then for $t \geq 0$ we have
$$v(t)=-i\int_0^t\Pi\big[|u_0+v|^2(u_0+v)\big](s)ds.$$
We take the $L^2$-norm and get with the help of \eqref{hypercontractivity}
$$\|v(t)\|_{L^2(\C)} \leq C_1t \|u_0+v\|^3_{L^6(\C)}\leq C_2 t(\|u_0\|^3_{L^6(\C)}+\|v\|^3_{L^6(\C)}) \leq C_3 t(\|u_0\|^3_{L^p(\C)}+\|v\|^3_{L^4(\C)}).$$
Therefore, by the conservation of the energy, we obtain $\|v(t)\|_{L^2(\C)} \leq Ct$ which is the second bound. The first bound follows from interpolation with the energy.
\end{proof}

\section{Long time results for the LLL equation}\label{Sect4}

\subsection{Bounds of Sobolev norms}

Recall that the $ L^{2,k}_\mathcal{E}$-norm is equivalent to the Sobolev {$\HH ^k(\C)$-norm} (see Lemma~\ref{lemEq}). Then we have the following bounds on the growth of such norms.

\begin{thm}
Let $k\geq 1$ be an integer and $u_0\in L^{2,k}_\mathcal{E}$. The equation \eqref{LLL} is globally well-posed in $L^{2,k}_\mathcal{E}$ and for any $t$,
\begin{equation}\label{bornehk}
\|u(t)\|_{L^{2,k}(\C)}\lesssim (1+|t|)^{\frac{k-1}2}.
\end{equation}
\end{thm}
\begin{proof} The global wellposedness in $L^{2,k}_\mathcal{E}$ easily follows from the  global wellposedness in $L^{2}_\mathcal{E}$. To get the bound, we compute
\begin{eqnarray*}
\frac{d}{dt}\int_{\C}|z|^{2k}|u(t,z)|^2dL(z)&=& 2\mathfrak{Re} \int_{\C} |z|^{2k} \ov{u}\partial_t u dL(z)\\
&=& 2\mathfrak{Im} \int_{\C} |z|^{2k} \ov{u}\Pi(|u|^2u) dL(z)\\
&=& 2\mathfrak{Im} \int_{\C} z|z|^{2(k-1)} \ov{zu}\Pi(|u|^2u) dL(z)\\
&\lesssim& \big\| zu\big\|_{L^2(\C)} \big\| |z|^{2k-1} \Pi(|u|^2u)\big\|_{L^2(\C)}.
\end{eqnarray*}
Next, by Proposition~\ref{prop22}
\begin{eqnarray*}
 \big\| |z|^{2k-1} \Pi(|u|^2u)\big\|_{L^2(\C)}  &\lesssim & \big\| \<z\>^{2k-1} |u|^2u\big\|_{L^2(\C)} \\
  &\lesssim & \big\| |z|^{2k-1} |u|^2u\big\|_{L^2(\C)} +C\big\|   |u|^2u\big\|_{L^2(\C)}\\
    &\lesssim &  \big\| zu\big\|_{L^2(\C)}  \big\| z^{k-1} u\big\|^2_{L^{\infty}(\C)}+C\big\|   u\big\|^3_{L^6(\C)}\\
       &\lesssim & \big\| zu\big\|_{L^2(\C)}  \big\| z^{k-1} u\big\|^2_{L^{2}(\C)}+C\big\|   u\big\|^3_{L^2(\C)},
\end{eqnarray*}
where the last line was obtained by the Carlen inequality \eqref{hypercontractivity} (using crucially that $u\in \tilde{\mathcal{E}}$ implies $z^ju \in \tilde{\mathcal{E}}$). Therefore, since $\big\| \<z\>u\big\|_{L^2(\C)}$ is uniformly bounded by conservation of $M$ and $P$, we get by interpolating that
\begin{equation}\label{bornzk}
\frac{d}{dt} \big\| \<z\>^ku\big\|^2_{L^2(\C)} \leq C \big\| \<z\>^ku\big\|^{2-\frac2{k-1}}_{L^2(\C)}.
\end{equation}
Then by a classical argument, \eqref{bornzk} implies $\|\<z\>^ku(t)\|_{L^2(\C)}\leq C (1+|t|)^{\frac{k-1}2}$, which in turn implies\;\eqref{bornehk} by Lemma \ref{lemEq}.
\end{proof}

\begin{rem}
It is interesting to compare this result to the bounds for the 2D  cubic Schr\"odinger equation
$$ i\partial_t u+\Delta_{\R^2}u -(x_1^2+x_2^2)u=|u|^2u, \quad (t,x_1,x_2)\in \R^3.$$
It is likely that with the method developed in \cite{PTV} one gets a bound $\lesssim (1+|t|)^{k-1}$.
\end{rem}

\subsection{Long time results for linear perturbations of the LLL equation}

Here we state some results concerning linear perturbations of the LLL equation which show, under generic assumptions, close-to-linear dynamics. In this setting,    the resonant structure of LLL is destroyed.

\subsubsection{KAM results for a perturbed equation}

 In the sequel, we consider the (non-local) perturbation of the (LLL) equation 
 \begin{equation}\label{LLLM}
i \partial_t u+ \nu \M u= \eps\Pi (|u|^2 u), \quad (t,z)\in \R\times \C,
\end{equation}
where $\nu, \eps>0$ are small and where  $\M$ is the (Hermite) multiplier, defined by $\M \phi_j=\xi_j \phi_j$ with $-1 \leq \xi_j\leq 1$. \medskip

Notice that $\M$ and $H$ commute and that we have the following conservation laws :
$$ \int_{\C}  |u(z)|^2dL(z),\qquad \int_{\C}\ov{u}Hu(z) dL(z),\qquad    \nu\int_{\C} \ov{u}\M u(z) dL(z)+\eps\int_{\C}  |u(z)|^4dL(z),$$
which are the $L^2$ and $L^{2,1}$ norms as well the Hamiltonian (there are other conservation laws).\bigskip

Using the commutation of  $\M$ and $H$, as well as the relation 
$$\e^{it H} \Pi\big(u_1 \ov{u_2} u_3 \big)=\Pi \big(\e^{it H}u_1 \,\ov{\e^{it H}u_2}\,\e^{it H}u_3\big),$$
which can be obtained by testing on $u_j=\phi_j$, we see that \eqref{LLLM} is equivalent to the equation (setting~{$v=\e^{it H}u$})
 \begin{equation} \label{cummut}
i \partial_t v +Hv+ \nu \M v= \Pi (|v|^2 v), \quad (t,z)\in \R\times \C.
\end{equation}~

The abstract KAM result  \cite[Theorem 2.3]{GT} can directly be applied to the equation \eqref{cummut} and hence \eqref{LLLM} (see also \cite[Theorem 6.6]{GT} for a similar statement for the Schr\"odinger equation with harmonic potential).
\begin{thm}\label{KAMLLL}
Let $n\geq 1$ be an integer and set $\mathcal{A}=[-1,1]^{n+1}$.   There exist $\eps_0>0$, $\nu_0>0$, $C_{0}>0$ and,  for each $\eps<\eps_0$, a  Cantor set $\mathcal{A}_{\eps}\subset \mathcal{A}$ of asymptotic  full measure when $\eps \to 0$, such that  for each $\xi\in \mathcal{A}_{\eps}$  and for each $C_{0}\eps\leq \nu<\nu_0$, the solution of 
\begin{equation}\label{nls2}
i \partial_t u+\nu \M u= \eps\Pi (|u|^2 u), \quad (t,z)\in \R\times \C,
\end{equation}
with initial datum
\begin{equation}\label{CI}
u_0(z)=\sum_{j=0}^n I_j^{1/2}e^{i\theta_j}\phi_j(z),
\end{equation}
with  $(I_0,\cdots,I_n)\subset (0,1]^{n+1}$ and $\theta\in \T^{n+1}$,
is quasi periodic with a quasi period $\omega^{\star}$ close to ${\omega_0=(2j+2)_{j=0}^n}$: $|\omega^{\star}-\omega_0|<C\nu$.\\
More precisely, when $\theta$ covers $\T^n$, the set of solutions of \eqref{nls2} with initial datum \eqref{CI} covers a $(n+1)$-dimensional torus which is invariant by \eqref{nls2}. Furthermore this torus is linearly stable.
\end{thm}

Notice that one already knew that the equation \eqref{LLLM} is globally well-posed for such initial conditions.

\subsubsection{Control of Sobolev norms  for a perturbed equation}  

  We define the   Hermite multiplier~$\M$ by   $\M\phi_j=m_j \phi_j\,,$
where $(m_j)_{j\in \N}$ is a bounded sequence of real numbers chosen in the following classes: for any $k\geq 1$, we define the class
$$
\W_k=\Big\{   (m_j)_{j\in \N} : \; m_j=\frac{\tilde m_j}{(j+1)^k} \;\mbox{ with }\; \tilde m_j\in [-1/2,1/2]\Big\}
$$
which is endowed with the product Lebesgue (probability) measure. Consider the problem
 \begin{equation}\label{LLLM0}
i \partial_t u+   \M u=  \Pi (|u|^2 u), \quad (t,z)\in \R\times \C.
\end{equation}

The following almost global existence result is proved in \cite[Theorem 1.1]{GIP}.
\begin{thm} \label{thFN}
Let $k,r\in \N$. There exists a set $\mathcal{B}_{k}\subset \W_{k}$ of  measure   $1$ such that if $(m_j)_{j\in \N}\in \mathcal{B}_{k}$  there exists   $s_{0}\in\N$ such that  
for any $s\geq s_{0}$,  there are $\eps_{0}>0$, 
$c>0$, such that for any 
$\eps\in (0,\eps_{0})$,  for any  $u_0 \in L^{2,s}_\mathcal{E}$ with 
$$\|u_0\|_{L^{2,s}(\C)} \leq \eps,$$  the equation \eqref{LLLM0} with initial datum\;$u_0$  has a unique global solution $u \in \mathcal{C}^{\infty}\big(\R, L^{2,s}_{\mathcal{E}}\big)$ and it satisfies 
$$ \|u(t)\|_{L^{2,s}(\C)}\leq 2\eps, \quad |t|\leq c \eps^{-r}.$$ 
\end{thm}

To prove this result, we apply \cite[Theorem 1.1]{GIP} to the equation $i \partial_t v+  Hv+ \M v=  \Pi (|v|^2 v)$, obtained with the change of unknown $v=\e^{it H}u$.

By the result of Lemma\;\ref{lemEq}, Theorem\;\ref{thFN} shows that if the initial condition is strongly localised in space, then the corresponding solution also remains localised for large times.

 
\section{Stationary waves and their decay: general results} \label{Sect5}

\subsection{Definition and decay result}

Stationary waves are naturally associated to the symmetries of the equation.

\begin{df}
An $M$-stationary wave is a solution of~\eqref{LLL} of the form 
$$u(t) = e^{-i \lambda t} u_0, \quad  \mbox{where $\lambda \in \mathbb{R}$, $u_0 \in \widetilde{\mathcal{E}}$}.$$
An $MP$-stationary wave is  a solution of~\eqref{LLL} of the form
$$u(t) = e^{-i \lambda t} u_0(e^{-i \mu t} \cdot), \quad \mbox{where $\lambda ,\mu \in \mathbb{R}$, $u_0 \in \widetilde{\mathcal{E}}$}.$$
\end{df}

The concept of $M$ and $MP$-stationary waves can immediately be extended to the space $\widetilde{\mathcal{E}}$. Note that $M$ and $MP$-stationary waves are given, respectively, by the solutions of
\begin{align*}
\lambda u = \Pi \big( |u|^2 u\Big),\qquad \lambda u + \mu \Lambda u = \Pi \big( |u|^2 u\Big).
\end{align*}

\begin{lem}\label{Q=0}
Assume that $u\in L^{2,1/2}$ is a $MP$-stationary wave with $\mu\ne 0$, then $Q(u)=0$.
\end{lem}

\begin{proof}
There exists $\psi \in L^{2,1/2}$ such that $u(t,z)=\e^{-i\lambda t}\psi(\e^{-i\mu t}z)$, with $\mu \ne 0$, and 
$$Q(u)(t)=\int_{\C}z|u(t,z)|^2dL(z)=\int_{\C}z|\psi(\e^{-i\mu t}z)|^2dL(z)=\e^{i\mu t}\int_{\C}z|\psi(z)|^2dL(z)=\e^{i\mu t}Q(u)(0).$$
By conservation of $Q(u)$, this implies that $Q(u)=0$.
\end{proof}

\begin{thm}
\begin{itemize}
\label{theoremdecay}
\item[(i)] Assume that $u = \sum_{n=0}^{+\infty} c_n \phi_n \in \widetilde{\mathcal{E}}$ is an $MP$-stationary wave such that $|c_n| \lesssim r^n$ for some $r<1$. Then, for any
$$
\gamma < \gamma_0 = \frac{1}{2} \frac{\log 2}{\log 3} \sim 0.315\dots,
$$
there holds $|c_n| \lesssim n^{-\gamma n}$.
\item[(ii)] Assume that $u(z) \in \widetilde{\mathcal{E}}$ is an $MP$-stationary wave such that $|u(z)| \in L^\infty$ and $u(z) \to 0$ as $z \to \infty$. Then for any 
$$
\eta>\eta_0 = \left( \frac{1}{2} + \frac{1}{2} \frac{\log 2}{\log 3} \right)^{-1} \sim 1.226\dots,
$$
there holds $|u(z)| \lesssim e^{|z|^\eta-\frac{1}{2}|z|^2}$.
\end{itemize}
\end{thm}

\begin{rem} The stationary waves exhibited in Theorem~\ref{thm41}, see also Appendix \ref{explicitstatwave}, give examples of:
\begin{itemize}
\item Finite energy stationary waves such that $c_n \sim \frac{r^n}{\sqrt{n!}}$ for any $r>0$ and $\sup_{|z|=\rho} |u(z)| \lesssim e^{-\frac{\rho^2}{2} + r \rho}$;
\item Infinite energy stationary waves $u \in L^\infty \setminus \cup_{\alpha>0} L^{\infty,\alpha}$ such that $c_n \sim \left\{ \begin{array}{ll} 0 & \mbox{if $n$ odd} \\ \frac{1}{n^{1/4}} & \mbox{if $n$ even} \end{array} \right.$.
\end{itemize}
These examples show that some of the conditions of the theorem are optimal; but they also suggest that stationary waves of finite energy might in general enjoy the bound $\sup_{|z|=\rho} |u(z)| \lesssim e^{-\frac{\rho^2}{2} + r \rho}$ for some $r$.
\end{rem}
\begin{cor}\label{coro35}
Let $u$ be an $MP$-stationary wave in $L^\infty$ such that $u(z) \to 0$ as $|z|\to \infty$, and let
$$
N(R) = \# \big\{ z \in \mathbb{C} \; \mbox{such that} \; u(z) = 0 \; \mbox{and} \; |z|<R \big\} .
$$
Then for any $\eta > \eta_0$, 
$$
\frac{N(R)}{R^\eta} \longrightarrow 0 \qquad \mbox{ as } \;\;R \to \infty.
$$
\end{cor}

\begin{rem}
Let $v \in \mathcal{E}$, then with the same  proof one obtains $N(R) \lesssim R^2$. This bound is sharp as shown by the two following examples : 
\begin{itemize}
\item Let $0<\delta <\frac12$ and set $\dis v(z)=\frac{\sin(\delta z^2)}{\delta z^2} e^{-|z|^2/2}$. Then $v \in \mathcal{E}$ and $N(R) \sim cR^2$. The zeros are located on the real and imaginary axes.
\item Let $0<\alpha<1$. The Weierstrass $\sigma_{\alpha}$-function associated to the lattice
$$\Lambda_{\alpha}=\Big\{\sqrt{\frac{\pi}{\alpha}}(m+in), \;\;m,n \in \Z\Big\},$$
satisfies $z \mapsto \sigma_{\alpha}(z) e^{-|z|^2/2} \in \mathcal{E}$ and vanishes exactly on $\Lambda_{\alpha}$, so that  $N(R) \sim cR^2$. See {\cite[Lemma 5.6, page 201]{Zhu}} for more details.
\end{itemize}
\end{rem}\medskip

\begin{proof}[Proof of Corollary \ref{coro35}] Write $u(z) = e^{-\frac{1}{2}|z|^2} f(z)$, where $f(z)$ is an entire function. Denote $\{ a_k \}$ the zeros of $f$. Assuming for simplicity that $f(0) \neq 0$, and provided that $f$ does not vanish on $\partial B(0,R)$, Jensen's formula gives
$$
\log |f(0)| = \sum_{|a_k| < R} \log \frac{|a_k|}{R} + \frac{1}{2\pi} \int_0^{2\pi} \log |f(Re^{i\theta})|\,d\theta.
$$
Denoting $N'(R) =\# \{ z \in \mathbb{C} \; \mbox{such that} \; u(z) = 0 \; \mbox{and} \; 0 <|z|< \frac{R}{2} \}$, the above clearly implies that
$$
 (\log 2) N'(R) \leq - \log |f(0)| + \frac{1}{2\pi} \int_0^{2\pi} \log |f(Re^{i\theta})|\,d\theta.
$$
By Theorem~\ref{theoremdecay}, for any $\eta >\eta _0$, $|f(z)| \lesssim_\eta e^{|z|^\eta }$. Combining this with the above inequality gives
$$
N'(R) \lesssim_\eta 1 +  R^\eta ,
$$
which leads to the desired result.
\end{proof}

\subsection{Proof of $(i)$ in Theorem~\ref{theoremdecay}} \underline{Step 1: a closer look at the $(c_n)$ equation.} The equation satisfied by $MP$- becomes, in $(c_n)$ coordinates
$$
(\lambda + \mu k) c_k = \frac{1}{2 \pi} \sum_{\substack{ \ell,m,n \geq 0 \\ k + \ell = m + n }} \frac{(k+\ell)!}{2^{k+\ell} \sqrt{k! \ell ! m! n!}} \overline{c_\ell} c_m c_n, \quad k\geq 0.
$$
For simplicity, and since this does not affect estimates, we shall take $\mu=0$ in the following. The above can also be written
$$
\lambda c_k = \frac{1}{2 \pi}\sum_{S=k}^{+\infty} \sum_{m=0}^S \sqrt{\frac{S!}{2^S k! (S-k)!}}\sqrt{\frac{S!}{2^S m! (S-m)!}}\overline{c_{S-k}} c_m c_{S-m}.
$$
By Stirling's formula, we can bound
$$
\sqrt{\frac{S!}{2^S k! (S-k)!}} \lesssim \psi\left( \frac{k}{S} \right)^S ,
$$
 where we denote, if $0<x<1$, $\displaystyle \psi(x) = \sqrt{\frac{1}{2 x^x (1-x)^{1-x}}}$. It will be important that $\psi(x)$ takes values in $(2^{-1/2},1]$, and is  equal to $1$ only if $x=\frac{1}{2}$.

Assuming that $|c_n| \lesssim r^n$ for some $r<1$, the above immediately implies that
$$
|c_k| \lesssim \sum_{S=k}^{+\infty} \sum_{m=0}^S \psi \left(\frac{k}{S} \right)^S \psi \left(\frac{m}{S} \right)^S  r^{2S - k}.
$$

\bigskip
\noindent
\underline{Step 2: the bootstrap argument.} Here we assume first that $|c_k| \leq C_r r^k$, for some $r<1$ to be determined and aim at obtaining a bound of the type $|c_k| \leq C_\rho \rho^k$, where $\rho$ depends on $r$ and $C_\rho$ on $C_r$. 
 
We fix $\kappa \in \left(\frac{1}{\sqrt{2}},1\right)$ and let $\epsilon \in \left(0,\frac{1}{2}\right)$ be such that $\psi(\epsilon) = \kappa$.  Observe that
$$\psi (x)\le \kappa  \quad \mbox{if $\left \vert x-\frac 12\right \vert \ge \frac 12-\epsilon $} \ .$$
 Splitting the sum above estimating $|c_k|$, we get
\begin{align*}
|c_k| & \lesssim C_r^3 \left[ \sum_{\substack{|2k-S|<2 \left(\frac{1}{2}-\epsilon\right) S  \\ |2m-S|<2 \left(\frac{1}{2}-\epsilon\right) S}}  r^{2S-k} +  \sum_{S=k}^{+\infty} \sum_{m=0}^S \psi(\epsilon)^S r^{2S-k} \right]  \\
& \lesssim C_r^3 k\left[ r^{\frac{1+\epsilon}{1-\epsilon}k} + (\kappa r)^k \right].
\end{align*}
We now assume that  $r$ is such that the second term in the above right-hand side dominates the first one, which corresponds to 
\begin{equation}\label{r0}
r\leq \kappa ^{\frac {1-\epsilon}{2\epsilon}} \ .
\end{equation}
Notice that, given $r<1$, \eqref{r0} is satisfied  if $\kappa <1$ is close enough to $1$.
Choosing furthermore any $\kappa' \in ( \kappa ,1)$, this gives
$$
|c_k| \lesssim (\kappa' r)^k.
$$
Thus we found that, for $r<1$ satisfying \eqref{r0}, $\kappa '\in (\kappa ,1)$, and for a constant $A>0$,
$$
|c_k| \leq C_r r^k \quad \implies \quad |c_k| \leq A (C_r)^3  (\kappa' r)^k.
$$
 Iterating this implication gives that
$$
|c_k| \leq B_n (\delta_n)^k \quad \mbox{where} \quad 
\left\{ \begin{array}{l} \delta_{n+1} = \kappa' \delta_n \\ B_{n+1} = A B_n^3 \end{array} \right. 
\quad \mbox{and} \quad \left\{ \begin{array}{l} \delta_{0} = r \\ B_{0} = C_r. \end{array} \right. 
$$
This implies in particular that, for any $n$, $k$, 
$$
|c_k| \lesssim (\kappa')^{nk} e^{C 3^n}.
$$
Choosing $n = \left [\frac{\log k}{\log 3}\right ]+1$, this gives the bound
$$
|c_k| \lesssim k^{-\gamma k}
$$
for any $\gamma < - \frac{\log \kappa'}{\log 3}.$ In particular, this implies 
$$|c_k|\lesssim r^k$$
for any $r\in (0,1)$. This means that \eqref{r0} is satisfied for every $\kappa \in (2^{-1/2}, 1)$. Applying again the same bootstrap argument, we obtain
$$
|c_k| \lesssim k^{-\gamma k}
$$
for any $\gamma <\gamma _0=\frac{\log2}{2\log3}.$ \medskip

\subsection{Proof of $(ii)$ in Theorem~\ref{theoremdecay}} ~\vspace{5pt}

\noindent\underline{Step 1: establishing Gaussian decay for $M$-stationary waves  in $z$ coordinates.} Without loss of generality, start with $u$, a function in $L^\infty$ going to zero at infinity, solving
$$
u = \Pi \big(|u|^2 u\big).
$$
Using first the Gaussian bound on the kernel of $\Pi$, and then elementary estimates, we get for $\kappa \in (0,1)$
\begin{align*}
|u(z)| & \lesssim \int e^{-\frac{1}{2}|w-z|^2} |u(w)|^3 \,dL(w) \\
& \leq \int_{|w| < \kappa |z|} e^{-\frac{1}{2}|w-z|^2} |u(w)|^3 \,dL(w) + \int_{|w| > \kappa |z|} e^{-\frac{1}{2}|w-z|^2} |u(w)|^3 \,dL(w) \\
& \lesssim e^{- \frac{(1-\kappa)^2}{3}|z|^2} + \sup_{|w| > \kappa |z|} |u(w)|^3.
\end{align*}
Setting $M_n = \sup_{|w|> \kappa^{-n}} |u(w)|$, this translates into
$$
M_n \leq C_0 e^{-\frac{(1-\kappa)^2}{3} \kappa^{-2n}} + C_0 M_{n-1}^3,
$$
for a constant $C_0$.

We now claim that $M_n < A e^{-\epsilon \kappa^{-2n}}$ for $n>n_0$, where $n_0$, $A$ and $\epsilon$ are positive constants to be determined. This will follow by induction if we can make sure that
\begin{equation*}
\left\{
\begin{array}{l}
M_{n_0} < A e^{-\epsilon \kappa^{-2n_0}} \\
C_0 e^{- \frac{(1-\kappa)^2}{3} \kappa^{-2n}} + C_0 A^3 e^{- 3 \kappa^2 \epsilon \kappa^{-2n}} < A e^{-\epsilon \kappa^{-2n}} \quad \mbox{for $n > n_0$},
\end{array}
\right.
\end{equation*}
which would follow from
\begin{equation}
\label{colvert}
\left\{
\begin{array}{l}
M_{n_0} < A e^{-\epsilon \kappa^{-2n_0}} \\
C_0 e^{- \frac{(1-\kappa)^2}{3} \kappa^{-2n}} < \frac{1}{2} A e^{-\epsilon \kappa^{-2n}} \quad \mbox{for $n>n_0$} \\
C_0 A^3 e^{- 3 \kappa^2 \epsilon \kappa^{-2n}} < \frac{1}{2} A e^{-\epsilon \kappa^{-2n}} \quad \mbox{for $n>n_0$},
\end{array}
\right.
\end{equation}
In order to make sure that these inequalities are satisfied, we choose $\kappa$, $n_0$, $A$ and $\epsilon$ as follows.
\begin{itemize}
\item First choose $\kappa = \frac{1}{\sqrt{3}}$ and $A < \frac{1}{\sqrt{2C_0}}$. This ensures that the third inequality in~\eqref{colvert} holds.
\item Next, pick $n_0$ so big that $C_0 e^{- \frac{(1-\kappa)^2}{6} \kappa^{-2n_0}} < \frac{A}{2}$ and $M_{n_0} < \frac{A}{2}$ (using that $M_{n} \to 0$ as $n \to \infty$ by hypothesis). This ensures that the second inequality in~\eqref{colvert} holds, provided $\epsilon < \frac{(1-\kappa)^2}{6}$.
\item Finally, choose $\epsilon \in \left(0, \frac{(1-\kappa)^2}{6}\right)$ so small that $\frac{1}{2} < e^{-\epsilon \kappa^{-2n_0}}$. Combined with $M_{n_0} < \frac{A}{2}$, this ensures that the first inequality in~\eqref{colvert} holds. 
\end{itemize}
Thus the claim holds, and we get that $|u(z)| \lesssim e^{-\sigma |z|^2}$ for some $\sigma>0$.\medskip

\underline{Step 1 bis: establishing Gaussian decay for $MP$-stationary waves in $z$ coordinates.}

Now we consider the equation $\lambda u+\mu \Lambda u =  \Pi(|u|^2u)$ with  $\mu \neq0$. Set $\alpha=\lambda/\mu$. \medskip

$\bullet$ Case $-\alpha \notin \N$. In this case, the equation is equivalent to 
$$u=\frac1\mu (\Lambda+\alpha)^{-1}\big[ \Pi(|u|^2u)\big].$$
Let us compute the kernel of $ (\Lambda+\alpha)^{-1}$. For all $n\in \N$, $ (\Lambda+\alpha)^{-1} \varphi_n= (n+\alpha)^{-1}\varphi_n$, then for $u\in F^2$, 
\begin{eqnarray*}
(\Lambda+\alpha)^{-1} u(z)&=& \sum_{n=0}^{+\infty} \frac1{n+\alpha} \big(\int_{\C} u(w)\ov{\varphi_n(w)}dL(w)\big) \varphi_n(z)\\
&=&   \int_{\C} u(w)K_{\alpha}(z,w) dL(w)
\end{eqnarray*}
with 
\begin{equation}\label{nk1}
K_{\alpha}(z,w) =   \sum_{n=0}^{+\infty} \frac1{n+\alpha} \varphi_n(z) \ov{\varphi_n(w)}  = \frac1{\pi}  \e^{-\frac{|z|^2}2-\frac{|w|^2}2} \sum_{n=0}^{+\infty} \frac{(z\ov{w})^n}{(n+\alpha) n!}.
\end{equation}
We claim that there exists $A\geq 0$ such that 
\begin{equation}\label{claim1}
|K_{\alpha}(z,w)|  \leq C(1+|zw|^{A})(\e^{\Re{(z\ov{w}})}+1)  \e^{-\frac{|z|^2}2-\frac{|w|^2}2}.
\end{equation}
Let $n_0$ be the smallest integer such that $n_0+\alpha>0$. Then 
 \begin{multline}
K_{\alpha}(z,w)   = \frac1{\pi}  \e^{-\frac{|z|^2}2-\frac{|w|^2}2} \sum_{n=0}^{n_0-1} \frac{(z\ov{w})^n}{(n+\alpha) n!}+ \frac1{\pi}  \e^{-\frac{|z|^2}2-\frac{|w|^2}2}  \int_0^1 t^{\alpha-1} \big(\sum_{n=n_0}^{+\infty} \frac{(t z\ov{w})^n}{n!}\big)dt\\
=\frac1{\pi}  \e^{-\frac{|z|^2}2-\frac{|w|^2}2} \sum_{n=0}^{n_0-1} \frac{(z\ov{w})^n}{(n+\alpha) n!}+ \frac1{\pi}  \e^{-\frac{|z|^2}2-\frac{|w|^2}2}  \int_0^1 t^{\alpha-1} \big(\e^{tz\ov{w}}-\sum_{n=0}^{n_0-1} \frac{(t z\ov{w})^n}{n!}\big)dt .\label{sum1}
\end{multline}
 If $|wz| \leq 1$, then from \eqref{nk1} we get $|K_{\alpha}(z,w)| \leq C \e^{-\frac{|z|^2}2-\frac{|w|^2}2}$. In the sequel we assume $|wz| \geq 1$. Then 
 \begin{multline}
  \int_0^1 t^{\alpha-1} \big|\e^{tz\ov{w}}-\sum_{n=0}^{n_0-1} \frac{(t z\ov{w})^n}{n!}\big|dt =\\
  \begin{aligned}
  &= \int_0^{\frac1{|wz|}} t^{\alpha-1} \big|\e^{tz\ov{w}}-\sum_{n=0}^{n_0-1} \frac{(t z\ov{w})^n}{n!}\big|dt+ \int_{\frac1{|wz|}}^{1} t^{\alpha-1} \big|\e^{tz\ov{w}}-\sum_{n=0}^{n_0-1} \frac{(t z\ov{w})^n}{n!}\big|dt \nonumber \\
  &= I_1+I_2.
  \end{aligned}
 \end{multline}
 In the first integral, we make the change of variables $s=t |wz|$ and get $I_1\leq C$. For the second, we get 
 \begin{equation*}
 I_2 \leq C(1+|zw|^{1-\alpha})(\e^{\Re{(z\ov{w}})}+|zw|^{n_0-1}+1).
 \end{equation*}
We also have the bound
 \begin{equation*}
 \big| \sum_{n=0}^{n_0-1} \frac{(z\ov{w})^n}{(n+\alpha) n!}\big| \leq C( |zw|^{n_0-1}+1).
 \end{equation*}
 As a conclusion, from \eqref{sum1} and the previous estimates we get \eqref{claim1}. \medskip
 
 $\bullet$ Case $-\alpha = n_0 \in \N$. 
 
 \begin{equation*}
K_{-n_0}(z,w) =   \sum_{n\neq n_0}\frac1{n-n_0} \varphi_n(z) \ov{\varphi_n(w)}  = \frac1{\pi}  \e^{-\frac{|z|^2}2-\frac{|w|^2}2} \sum_{n\neq n_0} \frac{(z\ov{w})^n}{(n-n_0) n!}.\label{nk}
\end{equation*}
 For $n\geq n_0+1$ we write $(n-n_0)^{-1}=\dis \int_0^{1}t^{n-n_0-1}dt$, and as previously we get 
  \begin{equation*}
K_{-n_0}(z,w) =    \frac1{\pi}  \e^{-\frac{|z|^2}2-\frac{|w|^2}2} \Big( \sum_{n=0}^{n_0-1} \frac{(z\ov{w})^n}{(n-n_0) n!} +\int_0^1 t^{-n_0-1} \big(\e^{tz\ov{w}}-\sum_{n=0}^{n_0} \frac{(t z\ov{w})^n}{n!}\big)dt\Big). 
\end{equation*}
 Similarly, there exists $A>0$ such that 
 \begin{equation}\label{claim2}
|K_{-n_0}(z,w)|  \leq C(1+|zw|^{A})(\e^{\Re{(z\ov{w}})}+1)  \e^{-\frac{|z|^2}2-\frac{|w|^2}2}.
\end{equation}\medskip
 
In the sequel, we assume that $|z|\geq 1$. We set $v=\Pi(|u|^2u)$. Then, by Step 1, 
 \begin{equation}\label{boot1}
 |z|^{3A}|v(z)| \leq \\
 C_0 e^{- \frac{(1-\kappa)^2}{3}|z|^2} + C_0\sup_{|w| > \kappa |z|} \big(|w|^{A} |u(w)|\big)^3.
 \end{equation}
 Then thanks to \eqref{nk1} and \eqref{claim2}
 \begin{multline*}
 |z|^A|u(z)| \leq C |z|^A\int_{\C}(1+|wz|^A) \e^{-\frac{|z|^2}2-\frac{|w|^2}2} |v(w)|dL(w)+\\
\begin{aligned}
&+C|z|^A \int_{\C}(1+|wz|^A) \e^{-\frac{|z-w|^2}2} |v(w)|dL(w)+C|z|^{n_0+A}\e^{-\frac{|z|^2}2} \nonumber\\
&=J_1+J_2+J_3.
\end{aligned}
 \end{multline*}
The term $J_3$ is the contribution of the mode $n_0$ in the case $\alpha=-n_0$, and we have $J_3\leq C\e^{-\frac{|z|^2}3}$. Then we clearly have $J_1\leq C\e^{-\frac{|z|^2}3}$. We write
\begin{eqnarray*}
J_2&=&C|z|^A \int_{|w|<\kappa |z|}(1+|wz|^A) \e^{-\frac{|z-w|^2}2} |v(w)|dL(w) +C|z|^A \int_{|w|>\kappa |z|}(1+|wz|^A) \e^{-\frac{|z-w|^2}2} |v(w)|dL(w)\nonumber\\
&\leq & C  e^{- \frac{(1-\kappa)^2}{3}|z|^2} + C\sup_{|w| > \kappa |z|} \big(|w|^{3A} |v(w)|\big).
\end{eqnarray*}
This implies that 
\begin{equation}\label{boot2}
 |z|^A|u(z)| \leq C  e^{- \frac{(1-\kappa)^2}{3}|z|^2} + C\sup_{|w| > \kappa |z|} \big(|w|^{3A} |v(w)|\big).
\end{equation}
 We set  $\dis M_n = \sup_{|w|> \kappa^{-n}} |w|^{A}|u(w)|$ and  $\dis N_n = \sup_{|w|> \kappa^{-n}} |w|^{3A}|v(w)|$, therefore 
 $$
N_n \leq C_0 e^{-\frac{(1-\kappa)^2}{3} \kappa^{-2n}} + C_0 M_{n-1}^3,
$$

 $$
M_n \leq C_0 e^{-\frac{(1-\kappa)^2}{3} \kappa^{-2n}} + C_0 N_{n-1}.
$$
We are now able to conclude as in Step 1 by induction (here we need to initialize $M_{n_0}$ and $M_{n_0+1}$).
\bigskip

\noindent
\underline{Step 2: bootstrapping in$(c_n)$ coordinates.} Since $|u(z)| \lesssim e^{-\sigma |z|^2}$ for some $\sigma >0$,  we can bound the coordinates $(c_n)$ of $u$ by
\begin{align*}
|c_n| & = \left| \int_\C u(z) \ov{\varphi_n(z)} \,dL(z) \right| 
\lesssim \frac{1}{\sqrt{n!}} \int_\C e^{-\left(\frac{1}{2} + \sigma \right) |z|^2} |z|^n \,dL(z) \\
& \lesssim \frac{\Gamma \left( \frac{n}{2} + 1 \right)}{\sqrt{n!} \left( \frac{1}{2} + \sigma \right)^{\frac{n}{2}+1}} ,
\end{align*}
where $\Gamma$ is Euler's Gamma function. By Stirling's formula,
$$
|c_n| \lesssim \frac{n^{1/4}}{(1+2\sigma)^{n/2}}.
$$
This means that $|c_n| \lesssim r^n$ for some $r\in (0,1)$. By $(i)$ in Theorem~\ref{theoremdecay}, we obtain that, for any $\gamma<\gamma_0$, $|c_n| \lesssim n^{-\gamma n}$.

\bigskip
\noindent
\underline{Step 3: back to $z$ coordinates.} Using that $|c_n| \lesssim n^{-\gamma n}$, for $\gamma<\gamma_0$, we get by Stirling's formula that
$$
|u(z)| = \left| \sum_{n=0}^{+\infty} c_n \frac{z^n}{\sqrt{\pi n!}} e^{-\frac{1}{2} |z|^2} \right| \lesssim \sum_{n=0}^{+\infty} n^{-\gamma n} \frac{|z|^n}{\sqrt{n!}} e^{-\frac{1}{2} |z|^2} \lesssim  \sum_{n=0}^{+\infty} {n^{-(\gamma+\frac{1}{2})n}} (e^{\frac12}|z|)^n e^{-\frac{1}{2} |z|^2}.
$$
By Young's inequality,
$$
|u(z)| \lesssim \left[  \sum_{k=0}^{+\infty} k^{-k} (2e^{\frac12}|z|)^{\frac{k}{\frac{1}{2} + \gamma}} \right]^{\frac{1}{2} + \gamma}e^{-\frac{1}{2} |z|^2} \lesssim e^{C | z|^{\frac{1}{\frac{1}{2} + \gamma}}-\frac{1}{2} |z|^2},
$$
which is the desired result.

\section{ Stationary waves with a finite number of zeros}\label{Sect6}

\subsection{The classification result}

\begin{thm}\label{thm41} \begin{itemize}

\item[(i)] $M$-stationary waves in $\mathcal{E}$ with a finite number of zeros and unit mass are given, modulo space and phase rotation, by $\phi_n^\alpha(z) e^{-i \lambda t}$ where
$$
\phi_n^\alpha(z) = R_{- \overline \alpha} (\phi_n)(z) = \frac{1}{\sqrt{\pi n!}} (z-\overline \alpha)^n e^{-\frac{|z|^2}{2}-\frac{|\alpha|^2}{2} + \alpha z} \quad 
\mbox{and\, $\left\{ \begin{array}{l} n \in \mathbb{N}, \alpha \in \mathbb{C} \\ \lambda = \frac{(2n)!}{\pi (n!)^2 2^{2n+1}} \end{array} \right.$}.
$$
They satisfy
$$
\mathcal{H}(\phi_n^\alpha) = \frac{1}{8\pi} \frac{(2n)!}{2^{2n} (n!)^2}, \;\;\;\;
M(\phi_n^\alpha) =1, \;\;\;\; P(\phi_n^\alpha) = n + |\alpha|^2, \;\;\;\; Q(\phi^\alpha_n) = \overline \alpha.
$$
\item[(ii)] Besides the $\phi_n^\alpha$, $MP$-stationary waves in $\mathcal{E}$ with a finite number of zeros and unit mass are given, modulo space and phase rotation, by
$\psi_b(e^{-i\mu t} z) e^{- i \lambda t }$, where
$$
\psi_b(z) = \frac{e^{-\frac{1}{2} \left( \frac{b}{1+b^2} \right)^2}}{\sqrt{\pi(1+b^2)}} \left( z - \frac{b(2+b^2)}{1+b^2}  \right) e^{-\frac{1}{2}|z|^2 + \frac{b}{1+b^2} z} \quad 
\mbox{and\, $\left\{ \begin{array}{l} b \in [0,\infty) \\ \lambda = \frac{1}{8 \pi (1+b^2)}\left(2b^2 + 1 + \frac{b^2}{1+b^2} \right) \\ \mu = - \frac{1}{8 \pi} \end{array} \right.$}.
$$
They satisfy
$$
\mathcal{H}(\psi_b) = \frac{1}{8 \pi} \left( 1 - \frac{1}{2(1+b^2)^2} \right), \;\;\;\; M(\psi_b) = 1, \;\;\;\; P(\psi_b) = \frac{1}{(1+b^2)^2}, \;\;\;\; Q(\psi_b) = 0.
$$
\item[(iii)] $M$-stationary waves in $\widetilde{\mathcal{E}} \setminus \mathcal{E}$ with a finite number of zeros are given, modulo space and phase rotation, by
$$
u(t) = A e^{- \frac{1}{2}|z|^2 + \frac{1}{2} z^2 + i  s z} e^{-i\lambda t}, \quad \mbox{where $A, s \in \mathbb{R}$, and $\lambda = \frac{A^2}{\sqrt{2}}$}.
$$
\item[(iv)] Besides the previous example, $MP$-stationary waves in $\widetilde {\mathcal{E}} \setminus \mathcal{E}$ with a finite number of zeros are given, modulo space and phase rotation, by
$$
u(t) = A (e^{-i\mu t} z + i r) e^{- \frac{1}{2}|z|^2 + \frac{1}{2}e^{- 2i\mu t} z^2} e^{-i\lambda t}, \;  \mbox{where }
\left\{ \begin{array}{l}
A \in \mathbb{R},\\
\lambda = \frac{1}{\sqrt{2}}\left(\frac{3}{2} + r^2 \right) A^2.\\
\mu = \frac{A^2}{\sqrt{2}}
\end{array} \right.
$$
\end{itemize}
\end{thm}

We postpone the proof of Theorem~\ref{thm41} to Paragraph~\ref{para44}, and refer to Section~\ref{dictionary} for the expression of these stationary waves in different coordinates.

\subsection{An invariant three--dimensional submanifold}
As a consequence of identifying $\psi_b$ in Theorem~\ref{thm41} as a stationary wave, we prove that the three--dimensional manifold 
\begin{equation}\label{oriole}
u(z)=(\lambda z+\mu)\, e^{\alpha z-\frac{|z|^2}{2}}\ ,\ \lambda \in {\mathbb C}^*\ ,\ \mu \in {\mathbb C}\ ,\ \alpha \in {\mathbb C}\ ,
\end{equation}
is invariant by the flow of \eqref{LLL}. This allows to recover results of \cite{BBCE} which were obtained by a direct calculation. 
\begin{prop}\label{prop43}
 For all $(\lambda, \mu,\alpha )\in \C^* \times \C \times \C$, there exists $(c,\varphi, a, b)\in \C ^*\times \T \times \C \times  \R$  such that
$$
(\lambda z + \mu) e^{\alpha z - \frac{1}{2} |z|^2} =c L_{\varphi}R_{a}\left [  \Big( z-\frac{b(2+b^2)}{1+b^2}  \Big) e^{- \frac{1}{2}|z|^2+\frac{b}{b^2+1} z}\right ].
$$
Thus, up to the symmetries of the equation, every solution to \eqref{LLL} corresponding to an initial condition of the form \eqref{oriole}, is a stationary wave.
\end{prop}
\begin{proof} It is clear that multiplication by $c\in \C^*$, action of $L_\varphi $ and of $R_a$ act on the manifold defined by \eqref{oriole}. With an operator $R_a$ we can reduce to the case when $\int_{\mathbb{C}}  z\vert u(z)\vert^2 \,dL(z)=0$. Then  the transform $L_{\varphi}$ allows to  reduce to the case $\alpha\in \R$, and by  multiplication by $c$ we can assume that $\lambda=1$. Hence, we are reduced to 
\begin{equation}\label{vinc}
0=\int_{\mathbb{C}}  z\vert u(z)\vert^2 \,dL(z) =  \pi \big( \ov{\mu}   {\alpha}^2 +{\mu} \alpha^2  +\alpha  |\mu|^2+  \alpha^3 +2\alpha + {\mu} \big) e^{\alpha^2},
\end{equation}
with $\alpha \in \R$ --- the calculation can be easily made using identity \eqref{bullfinch} below. We now claim that~\eqref{vinc} is satisfied if and only if there exists $b\in \R$ such that $\alpha=\frac{b}{b^2+1}$ and $\mu=-\frac{b(2+b^2)}{1+b^2}$, and this will complete the proof. 

Firstly, if \eqref{vinc} holds true, necessarily $\mu\in \R$, and we are led to  study the zeros of the second order polynomial $F(\mu)= \alpha \mu^2+ (2\alpha ^2+1)\mu+\alpha(2+\alpha^2).$ The critical value of $F$ is $\frac{1}{\alpha}(\alpha-\frac12)(\alpha+\frac12)$, thus $F$ admits a zero if and only if $-\frac12\leq \alpha \leq \frac12$. In this case, there exists $b\in \R$ such that $\alpha =\alpha(b)= \frac{b}{b^2+1}$, and we  obtain  that the zeros are $\mu_1(b)=-\frac{b(2+b^2)}{1+b^2}$ or $\mu_2(b)=-\frac{(2b^2+1)}{b(1+b^2)}$. This yields the claim, since $\alpha(b)=\alpha(1/b)$ and $\mu_2(b)=\mu_1(1/b)$.
\end{proof}

\subsection{Proof of the classification result}\label{para44}

We will simply solve the equations
$$
\lambda u = \Pi  \big( |u|^2 u\big) \quad \mbox{and} \quad \lambda u + \mu \Lambda  u = \Pi \big( |u|^2 u\big),
$$
over $\lambda, \mu \in \mathbb{R}$, and $u \in \widetilde{\mathcal{E}}$. First we need a result describing functions in $\widetilde{\mathcal{E}}$ with a finite number of zeros.

\bigskip
\noindent
\underline{Step 1: functions $u \in \widetilde{\mathcal{E}}$ with a finite number of zeros.} Write $u(z) = e^{-\frac{1}{2}|z|^2} f(z)$,  let $z_1 \dots z_k$ be the zeros of $f$ and define $P(z)=\prod_{j=1}^k(z-z_j)$. Then $\frac{f(z)}{P(z)}$ is an entire function which does not vanish, thus it can be written $\frac{f(z)}{P(z)} = e^{Q(z)}$, where $Q$ is an entire function. By the bounds on $u$, $Q$ is such that $\mathfrak{Re} Q(z) \lesssim \< z\>^2 $. The Borel-Caratheodory lemma implies that $|Q(z)|$ enjoys the same bounds, namely $|Q(z)| \lesssim  \<z\>^2$, which means, by the Liouville theorem, that $Q$ is a polynomial of degree at most 2. As a conclusion, any function satisfying the hypotheses of the proposition is of the type $u(z) = P(z) e^{Q(z)-\frac{1}{2}|z|^2}$, where $P$ and $Q$ are polynomials, and the degree of $Q$ is at most 2.

\bigskip
\noindent \underline{Step 2: $Q$ of degree 1, $\mu = 0$.} We look for $u$ of the form $P(z) e^{\alpha z - \frac{1}{2}|z|^2}$, with $\alpha \in \mathbb{C}$, and $P$ a polynomial, solving $\lambda u = \Pi |u|^2 u$. Recall the Gaussian integral identity
\begin{equation}
\label{gaussian}
\frac{1}{\pi} \int_\mathbb{C} e^{-2|w|^2 + aw + b\overline{w}} dL(w) = \frac{1}{2} e^{\frac{ab}{2}} \qquad \mbox{if $a,b \in \mathbb{C}$}.
\end{equation}
For any polynomial $P$ in $w, \overline{w}$, this implies
\begin{equation}
\label{bullfinch}
\frac{1}{\pi} \int_\mathbb{C} P(w,\overline{w}) e^{-2|w|^2 + aw + b\overline{w}} dL(w) = P(\partial_ a,\partial_b) \frac{1}{2} e^{\frac{ab}{2}}.
\end{equation}
Therefore
\begin{align}
\Pi (|u|^2 u)  (z) & = \frac{e^{-\frac{|z|^2}{2}}}{\pi} \int e^{-2|w|^2 + z\overline{w} + 2 \alpha w + \overline{\alpha w}} P(w)^2 \overline{P(w)}\,dL(w) \nonumber \\
& = \left. \frac{e^{-\frac{|z|^2}{2}}}{2} P(\partial_a)^2 \overline{P}(\partial_b)  e^{\frac{ab}{2}} \right|_{\substack{a = 2\alpha \\ b = \overline{\alpha} + z}}\nonumber \\
& =  \left. \frac{e^{-\frac{|z|^2}{2}}}{2} \overline{P}(\partial_b)   P\left(\frac{b}2\right)^2 e^{\frac{ab}{2}} \right|_{\substack{a = 2\alpha \\ b = \overline{\alpha} + z}}.\label{las}
\end{align}
Let $n\geq 0$ be the degree of $P$; the Taylor expansion of the polynomial $\ov{P}$ at point $a/2$ gives
\begin{equation*}
\ov{P}(\partial_b)=\ov{P}\left(\frac{a}2\right)+\ov{P}'\left(\frac{a}2\right)\left(\partial_b-\frac{a}2\right)+\cdots + \frac{1}{n!} \ov{P}^{(n)}\left(\frac{a}2\right)\left(\partial_b-\frac{a}2\right)^n.
\end{equation*}
Observe that $(\partial_b-\frac{a}2) e^{\frac{ab}{2}}=0$, then by \eqref{las} we get
\begin{eqnarray}
\Pi (|u|^2 u)  (z) &=&   \left. \frac12 e^{-\frac{|z|^2}{2}+\frac{ab}{2}}  \sum_{k=0}^n \frac{1}{k!} \ov{P}^{(k)}\left(\frac{a}2 \right)\partial^k_b\Big(  P\left(\frac{b}2\right)^2 \Big)\right|_{\substack{a = 2\alpha \\ b = \overline{\alpha} + z}}\nonumber\\
&=&   \left.\frac12 e^{-\frac{|z|^2}{2}+\alpha z+\vert \alpha\vert^2 }  \sum_{k=0}^n \frac{1}{k!} \ov{P}^{(k)}(\alpha)\partial^k_b\Big(  P\left(\frac{b}2\right)^2 \Big)\right|_{{  b = \overline{\alpha} + z}}.\label{sta}
\end{eqnarray}
If $u$ solves $\lambda u = \Pi (|u|^2 u)$, then the polynomial in $z$ appearing in the r.h.s. must have degree $n$. This is the case if and only if $\ov{P}^{(k)}(\alpha)=0$ for all $0\leq k\leq n-1$, hence $P$ takes the form $P(z)=A(z-\ov{\alpha})^n$, with $A \in \mathbb{C}$. Conversely, with \eqref{sta} we check that $u(z)=A(z-\ov{\alpha})^n e^{\alpha z - \frac{1}{2}|z|^2}$ is a stationary wave. There remains to normalize it to have mass one, giving $\phi^\alpha_n$.

\bigskip
\noindent \underline{Step 3: $Q$ of degree 1, $\mu \neq 0$.} Proceeding as in the previous step, for $u$ of the form $P(z) e^{\alpha z - \frac{1}{2}|z|^2}$, the equation $\lambda u + \mu \Lambda u = \Pi |u|^2 u$ is equivalent to the equality between polynomials
\begin{equation}
\label{redhawk}
\lambda P + \mu z P' + \alpha \mu z P =  \left.\frac12 e^{\vert \alpha\vert^2 }  \sum_{k=0}^n \frac{1}{k!} \ov{P}^{(k)}(\alpha)\partial^k_b\Big(  P\left(\frac{b}2\right)^2 \Big)\right|_{{  b = \overline{\alpha} + z}}.
\end{equation}
If $P$ has degree $n$, the polynomial on the l.h.s. has degree $n+1$, so this must be the degree of the polynomial on the r.h.s. This is only possible if $P^{(k)}(\alpha) = 0$ for $0 \leq k \leq n-2$, in other words, $P(z) = (z-\ov  \alpha)^n + \beta (z - \ov  \alpha)^{n-1}$ - taking without loss of generality the coefficient of $(z-\ov  \alpha)^n$ to be $1$.

With this form for $P$, we now expand the two sides of the above equation:
\begin{align*}
& LHS~\eqref{redhawk} = \Big[ (z-\overline{\alpha})^{n+1} \mu \alpha + (z-\overline{\alpha})^{n} (\mu n + \mu |\alpha|^2 + \beta \mu \alpha + \lambda) \\
& \qquad \qquad \qquad  \qquad  + (z-\overline{\alpha})^{n-1} (\lambda \beta+ \overline{\alpha} \mu n + \mu \beta (n-1) + \mu \beta |\alpha|^2) + (z-\overline{\alpha})^{n-2}\ov{\alpha} \mu \beta (n-1) \Big]\\[5pt]
& RHS~\eqref{redhawk} = \frac{1}{2} e^{ |\alpha|^2}  \left[ (z-\overline{\alpha})^{n+1} \frac{\overline{\beta} (2n)!}{2^{2n} (n+1)!} + (z-\overline{\alpha})^{n} \left( \frac{|\beta|^2 (2n-1)!}{2^{2n-2} n!} + \frac{(2n)!}{2^{2n} n!} \right) \right. \\
& \qquad \qquad \qquad  \qquad \left. +  (z-\overline{\alpha})^{n-1} \left( \frac{\beta(2n-1)!}{2^{2n-2}(n-1)!} + \frac{|\beta|^2 \beta (2n-2)!}{2^{2n-2} (n-1)!} \right) + (z-\overline{\alpha})^{n-2} \frac{\beta^2 (2n-2)!}{2^{2n-2} (n-2)!} \right]
\end{align*}
(where the last terms in the above expressions should be canceled if $n=1$). Identifying the coefficients and setting $(\mu',\lambda')=  2^{2n+1} (\mu,\lambda) e^{-|\alpha|^2}$ gives the system
\begin{subequations}
\begin{align}
\label{osprey1} &\mu' \alpha = \frac{\overline{\beta} (2n)!}{(n+1)!} \\
\label{osprey2} &\mu' n + \mu' |\alpha|^2 + \beta \mu' \alpha + \lambda' = \frac{4|\beta|^2(2n-1)!}{n!} + \frac{(2n)!}{n!} \\
\label{osprey3} &\lambda' \beta + \overline{\alpha} \mu' n + \mu' \beta (n-1) + \mu' \beta |\alpha|^2 = \frac{4\beta(2n-1)!}{(n-1)!} + \frac{4 |\beta|^2 \beta (2n-2)!}{(n-1)!} \\
\label{osprey4} &\mu' \ov{\alpha} \beta (n-1) = \frac{4\beta^2(2n-2)!}{(n-2)!}
\end{align}
\end{subequations}
(where the last line should be canceled if $n=1$). We now need to distinguish between the cases $n=1$ and $n>1$. \medskip

If $n=1$,~(\ref{osprey1}) gives $\mu' = \frac{\overline{\beta}}{\alpha}$  (unless $\alpha = 0$, but then $\beta = 0$ and we are back to step 2). Plugging this value of $\mu '$ in~(\ref{osprey2}) leads to $\lambda' = 3 |\beta|^2 + 2 - \frac{\overline{\beta}}{\alpha} - \overline{\alpha \beta}$, and  using this value of $\lambda'$ in~(\ref{osprey3}) gives the equation $|\beta|^2 \beta + 2 \beta + \frac{|\beta|^2}{\alpha}  {- \frac{\ov{\alpha \beta}}{\alpha}} = 0$. If $\beta=0$ we get the $M-$stationary wave $u(z)=A(z-\ov{\alpha}) e^{\alpha z - \frac{1}{2}|z|^2}$. Thus we can assume $\beta\neq 0$ and set $\alpha = a e^{i \varphi}$  and $\beta = b e^{i\psi}$ with $a,b>0$. We then observe that $X=e^{-i(\varphi+\psi)}$ satisfies $X^2-\frac{b}{a}X-(b^2+2)=0$. If $X\neq 1,-1$ this yields a contradiction because then $1=\vert X\vert^2=-(b^2+2)$. Finally we obtain $\beta = -b e^{-i\varphi}$ with $a=\frac{b}{b^2+1}$.\medskip

If $n \geq 2$, and $\beta \neq 0$,~\eqref{osprey4} gives that $\mu' = \frac{4 \beta (2n-2)!}{{\ov{\alpha}}(n-1)!}$, which implies first that $\alpha=ae^{i\varphi}$ and  $\beta=be^{-i\varphi}$ for some $a,b,\varphi\in \R$. Second, inserting this value of $\mu'$ in~\eqref{osprey1} leads to $4\frac{(2n-2)!}{(n-1)!}=\frac{(2n)!}{(n+1)!}$ which is impossible.

This leaves us with the stationary wave $u_b(z) = \left( z - \frac{b(2+b^2)}{1+b^2} \right) e^{-\frac{1}{2}|z|^2 + \frac{b}{1+b^2} z}$, which we need to normalize to have mass one. Using the identity
$$
\int e^{-|w|^2 + aw + c \overline w} \,dL(w) = \pi e^{ac},
$$
we obtain (noticing after the first equality that $-|z|^2 + 2 \mathfrak{Re}\left( \frac{b}{1+b^2} z \right) = - \left| z - \frac{b}{1+b^2} \right|^2 + \left( \frac{b}{1+b^2} \right)^2$)
\begin{align*}
\| u_b \|_{L^2}^2 & = \int \left| z - \frac{b(2+b^2)}{1+b^2} \right|^2 e^{-|z|^2 +  2 \mathfrak{Re}\left( \frac{b}{1+b^2} z \right)}\,dL(z) \\
& = e^{\left( \frac{b}{1+b^2} \right)^2} \int |z -  b|^2  e^{-|z|^2}\,dL(z) \\
&= e^{\left( \frac{b}{1+b^2} \right)^2}  \left. (\partial_a - b)(\partial_c - b) \pi e^{ac} \right|_{a=c=0} \\
 &=\pi(1+b^2)  e^{\left( \frac{b}{1+b^2} \right)^2}.
\end{align*}
This leads to the formula for $\psi_b = \frac{u_b}{\|u_b\|_{L^2}}$; proceeding similarly, one computes $\mathcal{H}(\psi_b)$ and $P(\psi_b)$. By Lemma~\ref{Q=0},  $Q(\psi_b)=0$.

\bigskip
\noindent \underline{Step 4: $Q$ of degree 2, $\mu = 0$.} In other words, we now look for solutions of $\lambda u = \Pi (|u|^2 u)$ of the type $P(z) e^{A z^2 + B z - \frac{1}{2}|z|^2}$, where $A,B \in \mathbb{C}$ and $P$ is a polynomial. We start from the following Gaussian integral.  For any complex numbers $a,b,c,d$ such that the integral converges absolutely,
\begin{eqnarray*}
\frac{1}{\pi} \int e^{-2 |w|^2 + aw + b\overline{w} + cw^2 + d\overline{w^2}}\,dL(w) &=& \frac{1}{2 \sqrt{1-cd}} e^{\frac{(1-cd)(a+b)^2 - (bc-ad+a-b)^2}{4(1-cd)(2-c-d)}}\\
&=&  \frac{1}{2 \sqrt{1-cd}} e^{\frac{da^2+cb^2+2ab}{4(1-cd)}}.
\end{eqnarray*}
Notice that the convergence of the integral implies $\Re (1-cd)>0$, so that the square root of $1-cd$ is defined classically. This identity implies, for a polynomial $P$ of $w$ and $\overline{w}$
\begin{equation}
\label{cormorant}
\frac{1}{\pi} \int e^{-2 |w|^2 + aw + b\overline{w} + cw^2 + d\overline{w^2}}P(w,\overline{w})\,dL(w) = P(\partial_a,\partial_b) \frac{1}{2 \sqrt{1-cd}} e^{\frac{da^2+cb^2+2ab}{4(1-cd)}}.
\end{equation}
Therefore,
\begin{align*}
\Pi (|u|^2 u) (z) & = \frac{e^{-\frac{|z|^2}{2}}}{\pi} \int e^{-2|w|^2 + z\overline{w} + 2 A w^2 + \overline{Aw^2} + 2Bw + \overline{B w}} P(w)^2 \overline{P(w)}\,dL(w) \\
& = \left. e^{-\frac{|z|^2}{2}} P(\partial_a)^2 \overline{P}(\partial_b) \frac{1}{2 \sqrt{1-cd}} e^{\frac{da^2+cb^2+2ab}{4(1-cd)}} \right|_{\substack{a = 2B \\ b = z+\overline{B} \\ c = 2A \\ d = \overline{A}}}. 
\end{align*}
For $u$ to be a stationary wave, the coefficients of $z^2$ and $z$ in $\frac{da^2+cb^2+2ab}{4(1-cd)}$, with $a = 2B$, $b = z+\overline{B}$, $c = 2A$, $d = \overline{A}$, must be $A$ and $B$ respectively. A small computation shows that the coefficients of\;$z^2$ agree if $A = \frac{A}{2(1-2|A|^2)}$, which gives $A=0$ (in which case we are back to step 2), or $|A| = \frac{1}{2}$. By rotation invariance, we can assume $A = \frac{1}{2}$; but then the coefficients of $z$ agree if $B = is$, with $s$ real.

Finally, observe that, if the degree of $P$ is $n$, the degree of the polynomial $Q$ such that $\Pi |u|^2 u = Q(z) e^{-\frac{1}{2} |z|^2+ \frac{1}{2} z^2 + is z}$, as determined by the formula above, is $3n$. Therefore, $n=0$.

\bigskip
\noindent \underline{Step 5: $Q$ of degree 2, $\mu \neq 0$.} Proceeding as in the previous step, any solution of $\lambda u + \mu \Lambda  u = \Pi |u|^2 u$ of the type $P(z) e^{A z^2 + B z - \frac{1}{2}|z|^2}$ is such that $A=0$, a case which we already examined, or $|A| = \frac{1}{2}$ and $B = is$, to which we now turn. Moreover, one realizes quickly that either $n=0$ (but this case has already been considered) or $n=1$, which we now examine. Therefore, write $u(z) = (z + \gamma) e^{-\frac{1}{2} |z|^2 + \frac{1}{2} z^2 + is z}$; computing using the above formula leads to
\begin{align*}
& (\lambda + \mu \Lambda )u(z) = \big[ \mu z^3 + (\mu is+\mu \gamma) z^2 + (\lambda + \mu + \mu\gamma is)z + \lambda \gamma \big] e^{-\frac{1}{2} |z|^2+ \frac{1}{2} z^2 + is z} \\
& \Pi( |u|^2 u)(z) = \frac{1}{\sqrt{2}} \left[ z^3 + \left(is + \overline{\gamma} + 2 \gamma \right) z^2 + \left( \frac{5}{2} + \gamma^2 + 2is\gamma + 2 |\gamma|^2 \right) z \right. \\
&\qquad \quad\qquad \qquad \qquad \left.+ \left(2\gamma + \frac{1}{2}is + \frac{1}{2} \overline{\gamma} + is\gamma^2 + |\gamma|^2 \gamma \right)\right]e^{-\frac{1}{2} |z|^2+ \frac{1}{2} z^2 + is z}.
\end{align*}
Identifying the coefficients of the powers of $z$, we find that $s=0$, $\gamma$ is pure imaginary: $\gamma = ir$, with\;$r$ real, $\mu = 1/\sqrt{2}$ and $\lambda = (\frac{3}{2} +r^2)/\sqrt{2}$.
\bigskip

\subsection{Construction of stationary waves by bifurcation from $\phi_0$}

While we only treat the case of $\phi_0$, identical arguments give bifurcation from the $\phi_n$, with $n \geq 1$. Recall the definition of the spaces $C_\eps$ given in \eqref{defC}.

\begin{prop}\label{bifurc}
For $k_0 \geq 2$ an integer, there exists, for $s \in \R$ sufficiently small, $MP$-stationary waves
$$
u = u_{k_0,s} = \sum_{\ell =0}^{+\infty} q_\ell (s) \varphi _{\ell k_0}=\varphi_0 + s \varphi_{k_0} + \mathcal{O}(s^2)
$$
(where $\mathcal{O}(s^2)$ is understood for the topology of $C_\eps$), which solve
$$
a u + b \Lambda u = 8\pi \Pi(|u|^2 u),
$$
with $a = 4$ and $\left|b - \frac{1}{k_0}\left(4 - \frac{8}{2^{k_0}} \right)\right| \lesssim s$.

Moreover, for all $\eps>0$, there exist $K_\eps>0$ and $s_\eps>0$ such that 
\begin{equation}\label{borneu}
|u(z)| \leq K_\eps e^{\eps |z|-\frac12|z|^2}
\end{equation}
for all $0 \leq s \leq s_{\eps}$.
\end{prop}

\begin{rem}
By Theorem~\ref{thm41}, for all $0 < s \leq s_{\eps}$, such a function has an infinite number of zeros. Indeed, 
none of the stationary waves listed in Theorem~\ref{thm41} has the property
$$u= \sum_{\ell =0}^{+\infty} q_\ell  \varphi _{\ell k_0}$$
for some $k_0\ge 2$. 
\end{rem}

\begin{proof}
Let $\eps >0$. Recall that $C_{\eps}$ is given by the norm $ \sup_{k\geq 0} \frac{\sqrt{k!}}{\eps^k} |c_k|$; abusing notations, we will identify the sequence $(c_n)$ and the corresponding function $\sum_n c_n \phi_n$, so that $C_\eps$ becomes a space of functions. We saw in Proposition~\ref{propC} that $(f,g,h) \mapsto \Pi (f \overline g h)$ is bounded from $C_{\eps}^3$ to $C_{\eps}$.

Restricting $C_{\eps}$ to indices which are multiples of $k_0$ gives
$$
C_{k_0,\eps} = \{ (c_k) \in C_{\eps} \; \mbox{such that} \; c_k = 0 \; \mbox{if $k$ is not a multiple of $k_0$} \}.
$$
We will apply the framework in Crandall-Rabinowitz~\cite[Theorem 1.7]{CR}.  Namely, let
$$
F(t,u) = 8\pi \Pi\left[ |\varphi_0 + u|^2 (\varphi_0 + u) \right] + t \Lambda(\varphi_0 + u) - 4(\varphi_0 + u).
$$
Observe that $F$ is a smooth function from $\mathbb{R} \times C_{k_0,\eps}$ to $(1+\Lambda) C_{k_0,\eps}$, such that
\begin{itemize}
\item $F(t,0) = 0$ for all $t$,
\item $\partial_t F(t,u) = \Lambda (\varphi_0 + u)$,
\item $\partial_u F(t,0) (\delta u) = 8 \pi \Pi(2 |\phi_0|^2 \delta u + \varphi_0^2  \overline{\delta u} ) + t \Lambda \delta u - 4 \delta u$ ; equivalently, in the $(c_k)$ coordinates: $\left[\partial_u F(t,0) (\delta u)\right]_k = \left( tk - 4 + \frac{8}{2^k} \right) \delta c_k + 4 \delta_{k,0} \overline{\delta c_k}$,
\item and finally $\partial_t \partial_u F(t,u)(\delta u) = \Lambda \delta u$.
\end{itemize}
Given $k_0 \geq 2$, we choose $t=t(k_0) = \frac{1}{k_0}\left(4 - \frac{8}{2^{k_0}} \right)$ such that $tk_0 - 4 + \frac{8}{2^{k_0}} = 0$ (notice that this determines $k_0$ uniquely if $k_0 \geq 4$, but that the same $t$ corresponds to $k_0 = 2$ and $k_0 =3$).

Since 
\begin{itemize}
\item $\operatorname{Ker} \partial_u F(t(k_0),0) = \operatorname{Span} \varphi_{k_0}$
\item $(\Lambda + 1) C_{k_0,\eps} / \operatorname{Ran}  \partial_u F(t(k_0),0)$ one-dimensional
\item $\partial_t \partial_u F(t,u)(\varphi_{k_0}) = \Lambda \varphi_{k_0}=k_0\varphi_{k_0} \notin \operatorname{Ran}  \partial_u F(t(k_0),0)$,
\end{itemize}
then~\cite[Theorem 1.7]{CR} applies, giving the existence result. 

The estimate \eqref{borneu} directly follows from the estimate $|c_k| \leq K_{\eps} \frac{\eps^k}{\sqrt{k !}}$.
\end{proof}

\section{Variational questions and stability properties}\label{Sect7}

\subsection{Maximizers of $\mathcal{H}$ for $M$ fixed}

The following proposition identifies the maximizers of the Hamiltonian for fixed mass. This result was already proved in~\cite[Theorem 2]{Carlen} via logarithmic Sobolev identities, and it can be deduced from \cite[Theorem 8.2]{FGH}, in the special case of the Bargmann--Fock space $\mathcal E$. We propose here a new, very elementary proof.

\begin{prop}\label{prop51}
If $u\in {\mathcal E}$, namely $u\in L^2(\C )$ and $u\, e^{\vert z\vert ^2/2}$ is entire, then $u\in L^4(\C )$, with the estimate
$$\Vert u\Vert _{L^4(\C )} ^4\le \frac 1{2\pi } \Vert u\Vert _{L^2(\C )} ^4\ .$$
Moreover, the above estimate is an equality if and only if 
$$u(z)=\lambda e^{\alpha z-\frac{\vert z\vert ^2}{2}}\ ,$$
for some $\lambda , \alpha \in \C $.
\end{prop}
\begin{proof}
The proof is inspired from the one of Lemma 1 of \cite{GGX}. Recall that
$$u=\sum _{n=0}^{+\infty} c_n\varphi _n,\quad \mbox{with} \quad \varphi _n(z)=\frac{1}{\sqrt{\pi n!}} z^n\, e^{-\frac{\vert z\vert ^2}{2}},$$
so that
$$\Vert u\Vert _{L^2(\C )}^2=\sum _{n=0}^{+\infty} \vert c_n\vert ^2\ .$$
We then classically write 
$$\Vert u\Vert _{L^4(\C )}^4=\Vert u^2\Vert _{L^2(\C )}^2\ ,$$
and observe that
\begin{eqnarray*}
u^2&=&\sum _{n,p\ge 0}c_nc_p\varphi _n\varphi _p=\sum _{n,p \geq 0}c_nc_p\left (\frac{(n+p)!}{n!p!}\right )^{1/2}\varphi _{n+p}\varphi _0\\
&=&\sum _{\ell =0}^{+\infty} \left (\sum _{n+p=\ell} c_nc_p\left (\frac{(n+p)!}{n!p!}\right )^{1/2}  \right )\varphi _\ell \varphi _0\ .
\end{eqnarray*}
We notice that the functions $\varphi _\ell \varphi _0$ are orthogonal in $L^2(\C )$ and that
$$\Vert \varphi _\ell \varphi _0�\Vert _{L^2(\C )}^2=\frac 1{\ell !{\pi }^2}\int _\C \vert z\vert ^{2\ell }\, e^{-2\vert z\vert ^2}\, dL(z)=
\frac 1{\pi 2^{\ell +1}}\ .$$
Consequently,
\begin{eqnarray}
\Vert u\Vert _{L^4(\C )}^4&=& \frac{1}{2\pi }\sum _{\ell =0}^{+\infty} \frac 1{2^{\ell }}\left \vert \sum _{n+p=\ell} c_nc_p\left (\frac{(n+p)!}{n!p!}\right )^{1/2}  \right \vert ^2\label{sst}\\
&\le &\frac{1}{2\pi }\sum _{\ell =0}^{+\infty} \frac{1}{2^\ell }\left ( \sum _{n+p=\ell } \frac{(n+p)!}{n!p!}\ \right ) \left ( \sum _{n+p=\ell} \vert c_nc_p\vert ^2 \right )=\frac{1}{2\pi }\Vert u\Vert _{L^2(\C )}^4\ ,\nonumber
\end{eqnarray}
where we used the Cauchy--Schwarz inequality. Furthermore, equality holds if and only if, for every $\ell \geq 0$, there exists $\gamma _\ell $ such that
$$\forall n=0,1,\dots ,\ell ,\  c_nc_{\ell -n}=\gamma _\ell \left (\frac{1}{n!(\ell -n)!}\right )^{1/2} \ ,$$
which is equivalent to 
$$\sqrt{n!}c_n \sqrt{(\ell -n)!}c_{\ell -n}=c_0\sqrt{\ell !}c_\ell\ ,\ n=0,1,\dots ,\ell ,$$
or
$$\sqrt {n!}c_n=\lambda  \alpha ^n$$
for some $\alpha ,\lambda \in \C $. Plugging this information into the formula, we get exactly
$$u(z)=\frac{\lambda }{\sqrt{\pi }} e^{\alpha z-\frac{\vert z\vert ^2}{2}} \ .$$
The proof is complete.
\end{proof}

Next, we aim at classifying maximizing sequences of $\mathcal H$ at $M$ fixed.  This will be achieved through the following profile decomposition lemma, in the spirit of \cite{S}, \cite{G}, \cite{BG}, \cite{MV}, \cite{ST}.

 \begin{lem}
Consider a sequence $(u_n)\in \mathcal E$ with $\|u_n \|_{L^2(\C)} = 1$. Then there exist $(v^j)\in  \mathcal E$, a sequence $(\alpha^j_n) \in \C$ with 
$$|\alpha^j_n-\alpha^k_n| \longrightarrow +\infty, \qquad n\longrightarrow +\infty,\quad j\neq k,$$
and  $(w_n^J)\in  \mathcal E$ with 
$$\limsup_{n\longrightarrow +\infty}\| w^J_n\|_{L^{\infty}(\C)} \longrightarrow 0, \qquad J\longrightarrow +\infty,$$
and such that we have, up to a subsequence, the decomposition 
$$u_n =\sum_{j=1}^J R_{\alpha^j_n}v^j+w^J_n,$$
and for all $J$
$$\sum_{j=1}^J \|v^j\|^2_{L^2(\C)}+ \limsup_{n\longrightarrow +\infty}\| w^J_n\|^2_{L^{2}(\C)} = 1.$$
\end{lem}

\begin{proof}
If $\| u_n\|_{L^{\infty}(\C)} \longrightarrow 0$, we can take $J=0$ and $w_n=u_n$. If not, then there exists $\epsilon_1>0$ such that, up to  a subsequence, and for $n$ large enough $\eps_1\leq \| u_n\|_{L^{\infty}(\C)} \leq 2\epsilon_1$ and there exists $\alpha^1_n \in \C$ such that $|u_n(\alpha^1_n)| \geq \epsilon_1$. We define $v_n^1=R_{-\alpha^1_n}u_n \in \mathcal E$ which satisfies 
\begin{equation}\label{lowb}
|v^1_n(0)| \geq \epsilon_1.
\end{equation}
 Next we write $v^1_n(z)=f^1_n(z) \e^{-|z|^2/2}$, where $f^1_n$ is entire. By the Carlen inequality, for all $z\in \C$,
$$|f^1_n(z) \e^{-|z|^2/2}  | \leq \|v^1_n\|_{L^{\infty}(\C)}=\|u_n\|_{L^{\infty}(\C)} \leq \frac1{\sqrt{\pi}} \|u_n\|_{L^{2}(\C)} \leq \frac1{\sqrt{\pi}}.$$
Therefore, for all $K>0$ and $n\geq 1$, we get 
$$|f^1_n(z)| \leq C_K, \quad |z|\leq K.$$
By the Montel theorem, there exists an entire function $f$ such that, up to a subsequence  $f^1_n \longrightarrow f^1$, uniformly on any compact of $\C$, and we can set $v^1(z)=f^1(z)\e^{-|z|^2/2} \in \mathcal E $. Moreover \eqref{lowb} implies $\| v^1\|_{L^{\infty}(\C)} \geq \epsilon_1$. Next, up to a subsequence $v^1_n \rightharpoonup v^1$ in $L^2(\C)$. We define $w^1_n=R_{\alpha^1_n}(v^1_n-v^1)$, thus $u_n=R_{\alpha^1_n}v^1+w_n^1$, and 
\begin{eqnarray*}
\|u_n\|^2_{L^2(\C)}&=&\|R_{\alpha^1_n}v^1\|^2_{L^2(\C)}+\|w_n^1\|^2_{L^2(\C)}+2 \Re\<R_{\alpha^1_n}v^1,w^1_n\>_{L^2(\C) \times L^2(\C)}\\
&=&\|v^1\|^2_{L^2(\C)}+\|w_n^1\|^2_{L^2(\C)}+2 \Re\<v^1,v^1_n-v^1\>_{L^2(\C) \times L^2(\C)}\\
&=&\|v^1\|^2_{L^2(\C)}+\|w_n^1\|^2_{L^2(\C)}+\kappa^1_n,
\end{eqnarray*}
with $\kappa^1_n \to 0$ since  $v^1_n \rightharpoonup v^1$ in $L^2(\C)$.

Now we repeat the procedure for the sequence $(w^1_n)$. Either $\| w^1_n\|_{L^{\infty}(\C)} \longrightarrow 0$ or there exists $\eps_2>0$ such that  $\eps_2\leq \| w^1_n\|_{L^{\infty}(\C)} \leq 2\epsilon_2$. Then similarly,
$$u_n =  R_{\alpha^1_n}v^1+ R_{\alpha^2_n}v^2+w^2_n,$$
for some  $v^2  \in \mathcal E$ such that $\|v^2\|_{L^{\infty}} \geq \epsilon_2$, $\alpha^2_n \in \C$ and $w^2_n  \in \mathcal E$. Similarly we check the almost orthogonality condition 
\begin{equation*}
\|u_n\|^2_{L^2(\C)}=\|v^1\|^2_{L^2(\C)}+\|v^2\|^2_{L^2(\C)}+\|w_n^2\|^2_{L^2(\C)}+\kappa^2_n, \quad \kappa^2_n \longrightarrow 0.
\end{equation*}
Let us prove that $|\alpha^1_n-\alpha^2_n| \longrightarrow +\infty$. From the relation $w_n^1=R_{\alpha^2_n}v^2+w^2_n$ we deduce that 
$$R_{-\alpha^1_n}w_n^1=R_{\alpha^2_n-\alpha^1_n}v^2+R_{-\alpha^1_n}w^2_n.$$ If we had that, for a subsequence $\alpha^1_n-\alpha^2_n \longrightarrow \ell \in \C$, this would be in contradiction with the fact that $R_{-\alpha^1_n}w_n^1$, $R_{-\alpha^2_n}w_n^2\rightharpoonup 0$ in $L^2(\C)$ and $v^2 \neq 0$.

As long as the remainder term does not converge to 0 in $L^{\infty}$, we construct a sequence $(v^j) \in \mathcal E$ such that 
$$u_n =\sum_{j=1}^J R_{\alpha^j_n}v^j+w^J_n,$$
with $\|v^j\|_{L^{\infty}} \geq \epsilon_j$ and 
\begin{equation*}
\|u_n\|^2_{L^2(\C)}= \sum_{j=1}^{J}\|v^j\|^2_{L^2(\C)}+\kappa^J_n,
\end{equation*}
with $\kappa^J_n \longrightarrow 0$ when $n\longrightarrow +\infty$. Then from Carlen and the previous line
\begin{eqnarray*}
\|u_n\|^2_{L^2(\C)}&\geq &\pi \sum_{j=1}^{J}\|v^j\|^2_{L^{\infty}(\C)}+\kappa^J_n\\
&\geq &\pi \sum_{j=1}^{J} \epsilon^2_j+\kappa^J_n,
\end{eqnarray*}
which implies that $\epsilon_J \longrightarrow 0$ and therefore $\| w^J_n\|_{L^{\infty}(\C)} \longrightarrow 0$.
\end{proof}

Here is a classical consequence of this profile decomposition.
 
 \begin{cor} 
\label{coromin}
Let $(u_n)$ be sequence in $L^2_\mathcal E$ such that 
 $$\Vert u_n\Vert_{L^2(\C)}^2\to \pi =\left \Vert {\rm e}^{-\frac {\vert z\vert ^2}{2}}\right \Vert _{L^2(\C)}^2\ ,\ \Vert u_n\Vert_{L^4(\C)}^4\to \frac{\pi }{2}=\left \Vert {\rm e}^{-\frac {\vert z\vert ^2}{2}}\right \Vert _{L^4(\C)}^4\ .$$
 Then, up to extracting a subsequence,  there exists $\beta _n\in \C$ and $\theta \in \T $ such that
 $$\left \Vert R_{\beta _n}u_n-{\rm e}^{i\theta }{\rm e}^{-\frac {\vert z\vert ^2}{2}}\right \Vert _{L^2(\C)}\rightarrow 0\ .$$
 \end{cor}
\begin{proof}
Up to extracting a subsequence, we apply the profile decomposition
$$u_n =\sum_{j=1}^J R_{\alpha^j_n}v^j+w^J_n,$$
with $$\limsup_{n\longrightarrow +\infty}\| w^J_n\|_{L^{\infty}(\C)} \longrightarrow 0, \qquad J\longrightarrow +\infty,$$
and
$$\sum_{j=1}^J \|v^j\|^2_{L^2}+ \limsup_{n\longrightarrow +\infty}\| w^J_n\|^2_{L^{2}} = \pi \ .$$
From H\"older's inequality, we infer
$$\limsup_{n\longrightarrow +\infty}\| w^J_n\|_{L^4(\C)} \longrightarrow 0, \qquad J\longrightarrow +\infty,$$
and, using 
$$|\alpha^j_n-\alpha^k_n| \longrightarrow +\infty, \qquad n\longrightarrow +\infty,\quad j\neq k,$$
we have 
$$\frac{\pi}{2}=\lim_{n\to \infty} \Vert u_n\Vert_{L^4}^4=\sum_{j=1}^{+\infty} \Vert v^j\Vert _{L^4}^4\ .$$
Now apply the Carlen inequality to each profile
$$\Vert v^j\Vert _{L^4}^4\leq \frac{1}{2\pi}\|v^j\|^4_{L^2}\ .$$
We obtain
\begin{eqnarray*}
\frac{\pi}{2}&=&\sum_{j=1}^{+\infty} \Vert v^j\Vert _{L^4}^4\leq \frac{1}{2\pi}\sum_{j=1}^{+\infty}  \|v^j\|^4_{L^2}\\
&\leq & \frac{1}{2\pi}\left (\sum_{j=1}^{+\infty}  \|v^j\|^2_{L^2}\right )^2\leq \frac{\pi}{2}\ .
\end{eqnarray*}
This implies that all the inequalities above are equalities, in particular there is only one $j$ --- say $j=1$--- such that $v^j\not =0$, and $w_n^J=w_n\to 0$ in $L^2$. In particular, $v^1$ is a minimizer of the $L^4-L^2$ Carlen inequality with mass $\pi $, so there exists $\alpha \in \C$ such that
$$v^1={\rm e}^{i\theta} R_{\alpha}\left ({\rm e}^{-\frac {\vert z\vert ^2}{2}}\right )\ ,$$
which yields the result by setting $\beta _n=-\alpha _n^1-\alpha $. 
\end{proof}

\subsection{Minimizers of $G_\mu = 8 \pi \mathcal{H} + \mu P$ for $M$ fixed} 

\begin{prop}[Local minimizers]
\label{proplocalmin}
\begin{itemize}
Consider for $\mu>0$ the minimization problem
$$
\min_{\substack{u \in \mathcal{E} \\ M(u) =1}} G_\mu \qquad \mbox{with} \qquad G_\mu =  8\pi \mathcal{H} + \mu P.
$$
\item[(i)] The function $\phi_0$ is a strict local minimizer (modulo the rotation of phase symmetry) if and only if  $\mu>\frac{1}{2}$.
\item[(ii)] The function $\phi_1$ is a strict local minimizer (modulo the rotation of phase symmetry) if and only if  $\frac{5}{32} < \mu < \frac{1}{2}$.
\item[(iii)]  If  $0<\mu <\frac{5}{32}$, then any local minimizer has an infinite number of zeros.
\item[(iv)] The function $\phi_k$, with $k \geq 2$ is not a local minimizer for any value of $\mu >0$.
\item[(v)] The function $\psi_b$, with $b>0$ is not a local minimizer for any value of $\mu\neq 1/2$.
\end{itemize}
\end{prop}

\begin{proof} $(i)$ Consider a deformation of $\phi_0$ at constant mass $M=1$ in $(c_k)$ coordinates: it is a function $s \mapsto (c_k(s))$ such that $c_k(0) = \delta_{k,0}$ and $\sum_{k=0}^{\infty} |c_k(s)|^2 =1$. Denoting with $\dot{}$ differentiation with respect to $s$, this last condition implies in particular that 
\begin{equation}
\label{albatros}
\mathfrak{Re} \dot{c_0}(0) = 0 \quad \mbox{and} \quad \mathfrak{Re} \ddot{c_0}(0) = - \sum_{k=0}^{+\infty} |\dot c_k(0)|^2.
\end{equation}
By using the phase rotation we can assume that $\mathfrak{Im} \dot{c_0} (0)= 0$, which gives $ \dot{c_0} (0) = 0$.
An immediate computation shows that (everything being evaluated at $s=0$)
$$
\left( \frac{d}{ds} \right)^2 G_\mu = 8 |\dot c_0|^2 + 4 \mathfrak{Re} \dot{c_0}^2 + 4 \mathfrak{Re} \ddot{c_0} + \sum_{n \geq 1} \left[ \frac{8}{2^n} + 2\mu n \right] |\dot c_n|^2.
$$
Making use of~\eqref{albatros}, this reduces to
$$
\left( \frac{d}{ds} \right)^2 G_\mu = \sum_{n \geq 1} \left[ \frac{8}{2^n} + 2\mu n - 4\right] |\dot{c}_n|^2.
$$
Therefore $\phi_0$ is a strict local minimizer iff $\frac{8}{2^n} + 2\mu n - 4 > 0$ for any $n \in \mathbb{N}$; but this is equivalent to $\mu > \frac{1}{2}$.

\bigskip 
\noindent
$(ii)$ Consider now a deformation of $\phi_1$ at constant mass $M=1$ in $(c_k)$ coordinates: 
$s \mapsto (c_k(s))$ such that $c_k(0) = \delta_{k,1}$ and $\sum_{k=0}^{\infty} |c_k(s)|^2 =1$. This implies in particular that 
\begin{equation}
\label{albatros2}
\mathfrak{Re} \dot{c}_1(0) = 0 \quad \mbox{and} \quad \mathfrak{Re} \ddot{c}_1(0) = - \sum_{k=0}^{+\infty} |\dot{c}_k(0)|^2.
\end{equation}
By using the phase rotation we can assume that $\mathfrak{Im} \dot{c_1} (0)= 0$, which gives $ \dot{c_1} (0) = 0$.
To simplify computations, introduce the following notation
$$
8 \pi \mathcal{H} = \sum_{\ell=0}^{\infty} \frac{1}{2^\ell} |S_\ell|^2, \;\; \mbox{with} \;\; S_\ell = \sum_{p+q = \ell} \sqrt{\frac{(p+q)!}{p! q!}} c_p c_q.
$$
Notice that, evaluated at $s=0$,
\begin{align*}
& S_\ell = 0 \;\; \mbox{and} \;\; \dot{S}_\ell = 2 \sqrt{\ell} \dot{c}_{\ell-1} \;\; \mbox{for $\ell \neq 2$} \\
& S_2 = \sqrt 2, \;\;\dot{S}_2 = 2 \sqrt{2} \dot{c}_1, \;\; \mbox{and} \;\; \ddot{S}_2 = 2 \sqrt{2} \dot{c}_1^2 + 2 \sqrt{2} \ddot{c}_1 + 4 \dot{c}_0 \dot{c}_2.
\end{align*}
Therefore,
\begin{align*}
&\left( \frac{d}{ds} \right)^2 G_\mu  = \sum_{n \neq 2} \frac{2}{2^n} |\dot{S}_n|^2 + \frac{1}{4} \left[ 2 |\dot{S}_2|^2 + 2 \mathfrak{Re} S_2 \ddot{S}_2 \right] + 2\mu \sum_{n \geq 1} n |\dot{c_n}|^2 + 2\mu \mathfrak{Re} \ddot{c_1} \\
& = \sum_{n \neq 2} \frac{n}{2^{n-3}} |\dot{c}_{n-1}|^2 + \frac{1}{4}\left[2 |2\sqrt{2} \dot{c_1}|^2 + 2\sqrt{2} \mathfrak{Re} (2\sqrt{2} \dot{c_1}^2 + 2\sqrt{2} \ddot{c_1} + 4 \dot{c_0} \dot{c}_2) \right] + 2\mu \sum_{n \geq 1} n |\dot{c}_n|^2 + 2\mu \mathfrak{Re} \ddot{c}_1.
\end{align*}
Making use of~\eqref{albatros2}, this reduces to
$$
\dots = (2-2\mu)|\dot{c}_0|^2  + (1+2\mu) |\dot{c}_2|^2 + 2 \sqrt{2} \mathfrak{Re} (\dot{c_0} \dot{c}_2) + \sum_{n \geq 3} |\dot{c}_n|^2 \left( \frac{n+1}{2^{n-2}} - 2 + 2\mu (n-1) \right).
$$
This (infinite dimensional) quadratic form in the $(\dot{c}_k)$ is positive if and only if
\begin{itemize}
\item The quadratic form $(x,y) \mapsto (2-2\mu)|x|^2  + (1+2\mu) |y|^2 + 2 \sqrt{2} \mathfrak{Re} (xy)$ is positive. This is the case if $0 < \mu < \frac{1}{2}$.
\item For any $n \geq 3$, $\frac{n+1}{2^{n-2}} - 2 + 2\mu (n-1) > 0$. This is the case for $\mu > \frac{5}{32}$.
\end{itemize}
This gives the desired result.

\bigskip
\noindent
$(iii)$ This will be a direct implication of $(iv)$ and $(v)$, combined with Theorem~\ref{thm41}.
\bigskip

\noindent
$(iv)$ We first show that $\phi_2$ cannot be a local minimizer. For $c_0,c_2,c_4 \in \C$ we compute
\begin{align*}
G_\mu (c_0 \phi_0 + c_2 \phi_2+c_4\phi_4) =&\mu(2|c_2|^2+4|c_4|^2)+|c_0|^4+ |c_0|^2|c_2|^2+ \frac38 |c_2|^4 + \frac14 |c_0|^2|c_4|^2\\
& \qquad\qquad+\frac{\sqrt{6}}4 \mathfrak{Re}( \ov{c_2}^2 c_0c_4)+\frac{15}{16} |c_2|^2|c_4|^2
+\frac{ 35}{128} |c_4|^4.
\end{align*}
Now, let $0<\epsilon<1/2$ and set $c_0=\epsilon$, $c_4=-\epsilon$ and $c_2=\sqrt{1-2\epsilon^2}$, then 
\begin{eqnarray*}
G_\mu (\epsilon \phi_0 + \sqrt{1-2\epsilon^2} \phi_2-\epsilon \phi_4) = \frac38 + 2\mu-\frac{4\sqrt{6}-7}{16}\epsilon^2 +\mathcal{O}(\epsilon^4) < G_\mu (  \phi_2 ),
\end{eqnarray*}
for $\epsilon>0$ small enough, which proves the result.

To show that $\phi_n$ cannot be a local minimizer for $n \geq 3$, observe that, if $0<\epsilon<1$,
$$
G_\mu (\sqrt{1-\epsilon^2} \phi_n + \epsilon \phi_0)
= \frac{(2n)!}{2^{2n}(n!)^2} + \mu n + \epsilon^2 \left[ \frac{1}{2^{n-2}} - \frac{(2n)!}{2^{2n-1}(n!)^2} - \mu n \right].
$$
Since $\frac{1}{2^{n-2}} - \frac{(2n)!}{2^{2n-1}(n!)^2} < 0$ for $n\geq 3$, $\phi_n$ cannot be a local minimizer. 
\medskip

\noindent
$(v)$  A direct computation shows that 
$$ G_\mu(\psi_b) = 1 + \left( \mu - \frac{1}{2} \right) \frac{1}{(1+b^2)^2}.$$
Then for $\mu \neq 1/2$,  a variation of $b$ may decrease this quantity, excepted in the case $b=0$, but then $\psi_0=\phi_1$ which is treated in point $(ii)$.
\end{proof}

Turning to the global minimization problem, observe that
\begin{align*}
&G_\mu(\phi_0) = 1 \\
&  G_\mu(\phi_1) = \frac{1}{2} + \mu \\
&  G_\mu(\psi_b) = 1 + \left( \mu - \frac{1}{2} \right) \frac{1}{(1+b^2)^2}.
\end{align*}
This implies in particular that $G_\mu(\phi_0) =  G_\mu(\phi_1) = G_\mu(\psi_b) = 1$ if $\mu = \frac{1}{2}$.

\begin{prop}[Global minimizers]\label{propglob} \begin{itemize}
\item[(i)] For any $\mu>0$, there exists a global minimizer of $G_\mu$ over $\{ u \in \mathcal{E}, M(u) = 1 \}$.
\item[(ii)] For $\mu\geq \sqrt 3-1$, $\phi_0$ is the unique global minimizer of $G_\mu$ over $\{ u \in \mathcal{E}, M(u) = 1 \}$.
\item[(iii)] For $\mu \in (0,\frac{5}{32})$, the global minimizer of $G_\mu$ has an infinity of zeros.
\end{itemize}
\end{prop}

\begin{proof} $(i)$ Consider a minimizing sequence $(u_n)$ in $\{ u \in \mathcal{E}, M(u) = 1 \}$ of $G_\mu$. Then $P(u_n)$ and $M(u_n)$ are uniformly bounded. On the one hand, by~\eqref{hypercontractivity}, $u_n$ is uniformly bounded in $B(0,R)$ for any $R$, and, by Cauchy's integral formula, so are all its derivatives; on the other hand, the $L^2$ mass  of $u_n$ on $B(0,R)^\complement$ is $\lesssim \frac{1}{R^2}$. Therefore, $(u_n)$ is precompact in $L^2$, and a subsequence converges to $u \in \mathcal{E}$ such that $M(u)=1$. By lower semi-continuity of $G_\mu$, we obtain that $u$ is a minimizer.
\bigskip

\noindent $(ii)$ By an homogeneity argument, the estimate $G_\mu \geq G_\mu (\phi_0)=1$ for $M(u)=1$ is equivalent to the following estimate for every $u$,
$$F_\mu (u)\geq 0\ ,\ F_\mu(u):=8\pi \mathcal H(u)+M(u)(\mu P(u)-M(u))\ .$$
The expression of $F_\mu (u)$ in variables $c_k$ reads
$$F_\mu =\sum_{\ell=0}^{+\infty} \frac{1}{2^\ell} \left |\sum_{p+q = \ell} \sqrt{\frac{(p+q)!}{p! q!}} c_p c_q \right |^2+\left (\sum_{k=0}^{+\infty} |c_k|^2\right )\left (\sum_{j=0}^{+\infty} (\mu j-1)|c_j|^2\right )\ .$$
Discarding the terms $\ell \ge 3$ in the first sum, and developing the others, we have
\begin{equation}\label{minFmu}
F_\mu \geq \mu |c_0|^2|c_1|^2+(2\mu -1)|c_0|^2|c_2|^2+\left (\mu -\frac 12\right )|c_1|^4 +(3\mu -2)|c_1|^2|c_2|^2+\sqrt 2\, {\Re}(\overline c_0c_1^2\overline c_2)+R_\mu \ ,
\end{equation}
where 
\begin{equation}\label{minRmu}
R_\mu :=(2\mu -1)|c_2|^4+\sum_{k=3}^{+\infty} (\mu k-1)|c_k|^4+\sum_{k=3}^{+\infty} |c_k|^2((\mu k-2)|c_0|^2+(\mu (k+1)-2)|c_1|^2+(\mu (k+2)-2)|c_2|^2)\ .
\end{equation}
Notice that $R_\mu \geq 0$ if $\mu \geq \frac 23$. Coming back to \eqref{minFmu}, we therefore observe that, for $\mu \geq \frac 23$,
\begin{eqnarray*}
F_\mu &\geq &\left (\mu -\frac 12\right )| c_1^2+\sqrt 2 c_0c_2|^2+\mu |c_0|^2|c_1|^2+(3\mu -2)|c_1|^2|c_2|^2+2\sqrt 2\,(1-\mu){\Re}(\overline c_0c_1^2\overline c_2)\\
&\geq & 0 ,
\end{eqnarray*}
if the remaining real quadratic form in $\overline c_0c_1, c_1\overline c_2$ is positive, which holds as soon as $$4\mu (3\mu -2)\geq 8(1-\mu )^2\ ,$$ namely $\mu ^2+2\mu -2\geq 0$, or $\mu \geq \sqrt 3-1$. Since $\sqrt 3-1\geq \frac 23$, this completes the proof of the inequality. If the equality holds for such $\mu $, then $R_\mu =0$, which means $c_k=0$ for $k\geq 2$, and $c_1^2+\sqrt 2 c_0c_2=0$, so $c_1=0$. Hence $u$ must be proportional to $\phi_0$.

\bigskip

\noindent $(iii)$ is an immediate consequence of Proposition \ref{proplocalmin}. 
\end{proof}

\begin{rem}
If one is interested about minimizing $G_\mu$ among even functions in $\mathcal{E}$, the situation is simpler:
\begin{itemize}
\item If $\mu > \frac12$, $\phi_0$ is the unique global minimizer.\vspace{3pt}
\item If $\mu < \frac12$, the global minimizer has an infinity of zeros.
\end{itemize}
The first claim follows from \eqref{minFmu} and \eqref{minRmu} by setting $c_{2n+1}=0$ for all $n\geq 0$. We turn to the second claim. Let  $\mu<1/2$, then by Theorem~\ref{thm41}, the only possible minimizers with a finite number of zeros are the $\phi_{2n}$, with $n\geq 1$. But the proof of  Proposition \ref{proplocalmin} ($iv$) shows that none of them is a local minimizer, among even functions.
\end{rem}
\bigskip

\begin{proof}[Proof of Theorem~\ref{thmmini}] The result follows from Proposition~\eqref{propglob} and a simple rescaling argument, setting $u(z)=\sqrt{h}v(\sqrt{h}z)$. \medskip

As in \cite{ABN}, denote by $\lambda$ a Lagrange multiplier associated to the problem \eqref{nrj} and denote by $e^h_{LLL}$ the global minimum of $E^{h}_{LLL}$. Then by~{\cite[Estimate (1.10)]{ABN}}, 
$$ \frac{2 \Omega_h}{3} \sqrt{\frac{2Na}{\pi}} <e^h_{LLL} \leq \lambda.$$
Therefore the  condition in \cite[Theorem 1.2]{ABN} is stronger than   the condition \eqref{56}.
\end{proof}

\subsection{Minimizers of $P$ for $\mathcal{H}$ and $M$ fixed}

Recall that for $u\in {\mathcal{E}}$
$$P(u) = \int_{\mathbb{C}} \Lambda  u(z) \overline{u(z)}\,dL(z)= \int_{\mathbb{C}} (\vert z\vert^2-1)\vert u(z)\vert^2 \,dL(z).$$ 
Given $M_0,H_0>0$, we study 
\begin{equation}\label{min}
\min_{\substack{\mathcal{H}(u)=H_0\\M(u)=M_0}}P(u).
\end{equation}
Recall that, by Proposition \ref{prop51}, for all $u \in \mathcal{E}$, $u \neq 0$, one has $8\pi \frac{\mathcal{H}(u)}{  M(u)^2} \leq 1$.
\begin{prop}\label{PropminP}
Fix  $M_0,H_0>0$ such that  $8\pi \frac{H_0}{  M_0^2} =\gamma$, where $\gamma\in (0,1/2)$ is such that $\gamma\neq \frac{(2n)!}{2^{2n}(n!)^2}$ for all $n\geq 1$. Then there exists $u\in \E$ which realises\;\eqref{min}. Moreover
\begin{enumerate}
\item[(i)] The function $u$ is an $MP$-stationary wave.
\item[(ii)] The function $u$ statisfies $\int_\C z\vert u(z)\vert^2dL(z)=0$.
\item[(iii)] The function $u$ has an infinite number of zeros in $\C$.
\end{enumerate}
\end{prop}

\begin{proof}
$(i)$ The Euler-Lagrange equation  corresponding to the problem \eqref{min} reads
\begin{equation*}
\Lambda u=\lambda u +\mu\Pi (|u|^2 u).
\end{equation*}
In order to get a $MP$-stationary wave, we have to check that $\mu\neq 0$. If $\mu=0$, then $u$ is an eigenfunction of $\Lambda$ in $\E$, thus $u(z)= \frac{z^n}{\sqrt{\pi n!}}  e^{- \frac{1}{2}|z|^2}  $ up to a constant factor. For such a $u$ we have   
\begin{equation*}
M(u)=1,\quad \quad \mathcal{H}(u)=\frac{1}{8\pi}\frac{(2n)!}{2^{2n}(n!)^2}\quad \mbox{and}\quad 8\pi \frac{\mathcal{H}(u)}{M^2(u)}=\frac{(2n)!}{2^{2n}(n!)^2},
\end{equation*}
which is excluded by assumption (by the way we check that the sequence $(2n)! /  ((n!)^2 2^{2n})$ is decreasing and equals $1/2$ when $n=1$). 
\medskip

$(ii)$ Let $\alpha \in \R$ and recall the definition \eqref{defR} of $R_{\alpha}$. Then $\mathcal{H}(R_\alpha u)=\mathcal{H}( u)$ and $M(R_\alpha u)=M( u)$, and we can check that 
\begin{equation*}
P(R_\alpha u)=P(u)-\alpha \int_\C(z+\ov{z})\vert u(z)\vert^2dL(z)+\alpha^2 \int_\C\vert u(z)\vert^2dL(z).
\end{equation*}
Thus, if $u$ realises the minimum in \eqref{min}, we get $\int_\C(z+\ov{z})\vert u(z)\vert^2dL(z)=0$, and the same argument with $R_{ i \alpha}$ then implies  $\int_\C z\vert u(z)\vert^2dL(z)=0$ --- see also Lemma \ref{Q=0}.
\medskip


$(iii)$ For this part, we rely on the classification in Theorem \ref{thm41} of the  $MP$-stationary waves which have a finite number of zeros. By the symmetries of the problem, we can assume that $A=1$ and $\varphi=0$.

$\bullet$ If  $u(z)=  (z-\ov{\alpha})^n e^{\alpha z-\frac{1}{2}|z|^2}$, then the condition $\int_\C z\vert u(z)\vert^2dL(z)=0$ implies $\alpha=0$. Thus we are reduced to the case    $ u(z)= A z^n e^{-\frac{1}{2}|z|^2}$ which is excluded, as we already observed.

$\bullet$ Assume that $u(z)= \Big(z-\frac{b(2+b^2)}{1+b^2}  \Big) e^{a z - \frac{1}{2}|z|^2} $ with $a=\frac{b}{1+b^2}$, $b\in \R$.  Then thanks to \eqref{bullfinch} we obtain for $v(z)=(z+\beta)e^{\alpha z- \frac{1}{2}|z|^2}$ with $\alpha,\beta\in \R$
\begin{equation*}
\frac{1}{\pi} \int_\mathbb{C} z\vert v(z)\vert ^2  dL(w) = (2\alpha+\beta+\alpha\beta^2+2\alpha^2 \beta+\alpha^3)  e^{{\beta^2}}.
\end{equation*}
Therefore, by \eqref{vinc}, for all $b\in \R$
\begin{equation*}
  \int_\mathbb{C} z\vert u(z)\vert ^2  dL(w) =0, \quad \mbox{when}\quad \alpha =\frac{b}{1+b^2}, \;\;\beta =-\frac{b(2+b^2)}{1+b^2} .
\end{equation*}
With this choice $R_{-\beta}u(z)=c_b ze^{-b z - \frac{1}{2}|z|^2}$ with $c_b=e^{-\frac{b^4(2+b^2)}{2(1+b^2)^2}}$, and thus 
\begin{equation*}
M(u)=\pi c_b (1+b^2)e^{b^2}, \quad \quad    \mathcal{H}(u)=\frac{\pi}{4} c^2_b (1+4b^2+2b^4)e^{2b^2}, 
\end{equation*}
which implies 
\begin{equation}\label{quot}
8\pi \frac{\mathcal{H}(u)}{M^2(u)}=\frac{1+4b^2+2b^4}{2(1+b^2)^2}\in [\frac12,1).
\end{equation}

 Hence if we choose $M_0,H_0$ as in the proposition, the stationary solution we find has an infinite number of zeros, by Theorem \ref{thm41}. 
\end{proof}
\medskip

Now we consider the minimizing problem \eqref{min}, when  $8\pi \frac{H_0}{  M_0^2} =\gamma\in [1/2,1)$. In this case, by~\eqref{quot}, there exists a unique $b\geq 0$ such that $8\pi \frac{ \mathcal{H}(\psi_b)}{  M(\psi_b)^2}=\gamma$, and we have

\begin{prop}[Local minimizers]\label{proplocalmin2}
Let $b\geq 0$ and consider  the minimization problem
\begin{equation*} 
\min_{\substack{\mathcal{H}(u)=H_0\\M(u)=1}}P(u),
\end{equation*}
with $H_0=\mathcal{H}(\psi_b) = \frac{1}{8 \pi} \left( 1 - \frac{1}{2(1+b^2)^2} \right)$. Then the function $\psi_b$ is a strict local minimizer (modulo the rotation of phase and the rotation of space symmetries).
\end{prop}

\begin{proof}
Let $b\geq 0$ and recall that $\psi_b(z) = \frac{e^{-\frac{1}{2} \left( \frac{b}{1+b^2} \right)^2}}{\sqrt{\pi(1+b^2)}} \left( z - \frac{b(2+b^2)}{1+b^2}  \right) e^{-\frac{1}{2}|z|^2 + \frac{b}{1+b^2} z} $. We set $\alpha=\frac{b}{1+b^2}$. Consider a deformation of $\psi_b$ at constant mass $M=1$ and constant Hamiltonian ${\mathcal{H}=\mathcal{H}(\psi_b)}$ in   coordinates given by the $(\phi_n^{\alpha})_{n \geq 0}$. We have
$$v(s,z) =\sum_{n=0}^{+\infty}c_n(s) \phi_n^{\alpha}(z), \quad \alpha=\frac{b}{1+b^2},$$
with 
\begin{equation}\label{defc}
c_0(0)=-\frac{b}{\sqrt{1+b^2}}, \quad c_1(0)=\frac{1}{\sqrt{1+b^2}}, \quad c_n(0)=0\; \text{ for }\; n\geq 2.
\end{equation}
The condition  $\sum_{n=0}^{+\infty} |c_n(s)|^2 =1$ gives after differentiation 
\begin{equation}\label{cond1}
-b \mathfrak{Re} (\dot{c}_{0}(0))+\mathfrak{Re} (\dot{c}_{1}(0))=0,
\end{equation}
and differentiating a second time
\begin{equation}\label{dcc1}
-\frac{b}{\sqrt{1+b^2}} \mathfrak{Re} (\ddot{c}_{0}(0))+\frac{1}{\sqrt{1+b^2}}\mathfrak{Re} (\ddot{c}_{1}(0))+\sum_{n=0}^{+\infty} |\dot c_n(0)|^2=0.
\end{equation}

Next
\begin{eqnarray*}
 \frac{d}{ds}  H(u)= 0&=&\mathfrak{Re} \sum_{\substack{ k,\ell,m,n \geq 0 \\ k + \ell = m + n }}   \frac{(k + \ell)!}{2^{k+\ell} \sqrt{k! \ell ! m ! n!}} \overline{c_k} \overline{c_\ell} \dot{c}_m c_n \\
 &=& (c^3_0+ c_0 c^2_1)\mathfrak{Re} \dot{c}_0+(c^2_0c_1+ \frac12 c^3_1)\mathfrak{Re} \dot{c}_1+\frac{\sqrt{2}}{4}c_0c^2_1 \mathfrak{Re} \dot{c}_2,
\end{eqnarray*}
hence
\begin{equation}\label{cond2}
-b(1+b^2) \mathfrak{Re} \dot{c}_0 +(b^2+\frac12)\mathfrak{Re} \dot{c}_1-\frac{\sqrt{2}}4b\mathfrak{Re} \dot{c}_2=0.
\end{equation}

Define 
$$u(s,z)=R_{\alpha}v(s,z) =\sum_{n=0}^{+\infty}c_n(s) \phi_n(z),$$

\begin{eqnarray*}
P(v)&=&P(R_{-\alpha} u)=P(u)+2 \alpha \mathfrak{Re} (Q(u))+\alpha^2\\
&=& \sum_{n=0}^{+\infty}n |c_n|^2 +\frac{2b}{1+b^2}  \mathfrak{Re} \Big(  \sum_{n=0}^{+\infty}\sqrt{n+1} c_n \ov{c_{n+1}}  \Big)+\Big( \frac{b}{1+b^2} \Big)^2.
\end{eqnarray*}
Firstly, one checks that $\frac{d}{ds}  P(v)=0$ at $s=0$, thanks to \eqref{cond1} and \eqref{cond2}. \medskip

An immediate computation shows that (everything being evaluated at $s=0$)
\begin{multline}
\left( \frac{d}{ds} \right)^2 P(v)=\\
\begin{aligned}\label{defpp}
 &= 2 \sum_{n=0}^{+\infty} n|\dot c_n|^2+2 c_1  \mathfrak{Re} \ddot{c_1} +\frac{2b}{1+b^2} \Big(  c_1  \mathfrak{Re} \ddot{c_0}+2   \mathfrak{Re}\Big(  \sum_{n=0}^{+\infty}\sqrt{n+1} \dot{c_n} \ov{\dot{c}_{n+1}}  \Big)+ c_0  \mathfrak{Re} \ddot{c_1}+\sqrt{2} c_1  \mathfrak{Re} \ddot{c_2}\Big)\\
&= 2 \sum_{n=0}^{+\infty} n|\dot c_n|^2+  \frac{4b}{1+b^2} \mathfrak{Re}\Big(  \sum_{n=0}^{+\infty}\sqrt{n+1} \dot{c_n} \ov{\dot{c}_{n+1}}  \Big)   +  
\frac{2}{(1+b^2)^{\frac32}} \big(  b \mathfrak{Re} \ddot{c_0} +    \mathfrak{Re} \ddot{c_1} +   \sqrt{2}  b\mathfrak{Re} \ddot{c_2}\Big). 
\end{aligned}
\end{multline}
The condition $\left( \frac{d}{ds} \right)^2H=0$ at $s=0$ gives
\begin{multline*}
 \mathfrak{Re} \sum_{\substack{ k,\ell,m,n \geq 0 \\ k + \ell = m + n }}   \frac{(k + \ell)!}{2^{k+\ell} \sqrt{k! \ell ! m ! n!}} \overline{c_k} \overline{c_\ell} \dot{c}_m \dot{c}_n+ \mathfrak{Re} \sum_{\substack{ k,\ell,m,n \geq 0 \\ k + \ell = m + n }}   \frac{(k + \ell)!}{2^{k+\ell} \sqrt{k! \ell ! m ! n!}} \overline{c_k} \overline{c_\ell} \ddot{c}_m c_n \\
 \qquad \qquad  \qquad  \qquad  \qquad  \qquad   +2\mathfrak{Re} \sum_{\substack{ k,\ell,m,n \geq 0 \\ k + \ell = m + n }}   \frac{(k + \ell)!}{2^{k+\ell} \sqrt{k! \ell ! m ! n!}} \overline{c_k} \overline{\dot{c}_\ell} \dot{c}_m c_n=\\
 = c^2_0\mathfrak{Re} (\dot{c}^2_0)+2c_0c_1\mathfrak{Re} (\dot{c}_0\dot{c}_1)+ \frac12 c^2_1\mathfrak{Re} (\dot{c}^2_1) + \frac{\sqrt{2}}2c^2_1 \mathfrak{Re} ( \dot{c}_0\dot{c}_2)\\
 + (c^3_0+ c_0 c^2_1)\mathfrak{Re} \ddot{c}_0+(c^2_0c_1+ \frac12 c^3_1)\mathfrak{Re} \ddot{c}_1+\frac{\sqrt{2}}{4}c_0c^2_1 \mathfrak{Re} \ddot{c}_2\\
 +2c^2_0   \sum_{m=0}^{+\infty}\frac{1}{2^m}|\dot{c}_m|^2+c^2_1     \sum_{m=0}^{+\infty}\frac{(m+1)}{2^m}|\dot{c}_m|^2 +2c_0c_1 \mathfrak{Re}  \sum_{m=0}^{+\infty}\frac{\sqrt{m+1}}{2^m}\dot{c}_m \ov{\dot{c}_{m+1}}  =0.
 \end{multline*}
 Then by \eqref{defc}, the previous line reads
 \begin{equation}\label{dcc2}
 \frac1{(1+b^2)^{\frac12}} \Big(   -b(1+b^2)\mathfrak{Re} \ddot{c}_0+(b^2+\frac12)\mathfrak{Re} \ddot{c}_1-\frac{\sqrt{2}}{4}b \mathfrak{Re} \ddot{c}_2 \Big)
 +  \Sigma  =0,
\end{equation}
with 
 \begin{multline*} 
\Sigma= b^2\mathfrak{Re} (\dot{c}^2_0)-2b\mathfrak{Re} (\dot{c}_0\dot{c}_1)+ \frac12 \mathfrak{Re} (\dot{c}^2_1) + \frac{\sqrt{2}}2  \mathfrak{Re} ( \dot{c}_0\dot{c}_2)\\
 \\
 +      \sum_{m=0}^{+\infty}\frac{(2b^2+m+1)}{2^m}|\dot{c}_m|^2 -2b \mathfrak{Re}  \sum_{m=0}^{+\infty}\frac{\sqrt{m+1}}{2^m}\dot{c}_m \ov{\dot{c}_{m+1}} .
\end{multline*}

 We simplify the last term in \eqref{defpp}. Thanks to \eqref{dcc1} and \eqref{dcc2} we obtain
  \begin{multline*}
 \frac{1}{(1+b^2)^{\frac12}} \big(  b \mathfrak{Re} \ddot{c_0} +    \mathfrak{Re} \ddot{c_1} +   \sqrt{2}  b\mathfrak{Re} \ddot{c_2}\Big)= -(4b^2+3)  \sum_{n=0}^{+\infty} |\dot{c}_n|^2+4\Sigma=\\
 =-(4b^2+3)  \sum_{n=0}^{+\infty} |\dot{c}_n|^2+4\Big[ b^2\mathfrak{Re} (\dot{c}^2_0)-2b\mathfrak{Re} (\dot{c}_0\dot{c}_1)+ \frac12 \mathfrak{Re} (\dot{c}^2_1) + \frac{\sqrt{2}}2  \mathfrak{Re} ( \dot{c}_0\dot{c}_2)  \Big]\\
 +4\Big[        \sum_{m=0}^{+\infty}\frac{(2b^2+m+1)}{2^m}|\dot{c}_m|^2 -2b \mathfrak{Re}  \sum_{m=0}^{+\infty}\frac{\sqrt{m+1}}{2^m}\dot{c}_m \ov{\dot{c}_{m+1}}     \Big].
 \end{multline*}
As a consequence, from \eqref{defpp} we get
 \begin{multline}\label{517}
\left( \frac{d}{ds} \right)^2 P(v)=
2 \sum_{n=0}^{+\infty} n|\dot c_n|^2+  \frac{4b}{1+b^2} \mathfrak{Re}\Big(  \sum_{n=0}^{+\infty}\sqrt{n+1} \dot{c_n} \ov{\dot{c}_{n+1}}  \Big)  \\ +  
\frac{2}{(1+b^2)} \Big[    -(4b^2+3)  \sum_{n=0}^{+\infty} |\dot{c}_n|^2+4\big( b^2\mathfrak{Re} (\dot{c}^2_0)-2b\mathfrak{Re} (\dot{c}_0\dot{c}_1)+ \frac12 \mathfrak{Re} (\dot{c}^2_1) + \frac{\sqrt{2}}2  \mathfrak{Re} ( \dot{c}_0\dot{c}_2)  \big)\Big]\\
 +\frac{8}{(1+b^2)}\Big[        \sum_{m=0}^{+\infty}\frac{(2b^2+m+1)}{2^m}|\dot{c}_m|^2 -2b \mathfrak{Re}  \sum_{m=0}^{+\infty}\frac{\sqrt{m+1}}{2^m}\dot{c}_m \ov{\dot{c}_{m+1}}  \Big]\\
 := \frac{2}{1+b^2}\big(\mathcal{Q}_1+\mathcal{Q}_2+\mathcal{Q}_3\big),
 \end{multline}
 where 
  \begin{multline*}
 \mathcal{Q}_1 =   (4b^2+1)|\dot c_0|^2+  (b^2+2)|\dot c_1|^2+ 2|\dot c_2|^2+ 4b^2\mathfrak{Re} (\dot{c}^2_0)+2\mathfrak{Re} (\dot{c}^2_1)\\
 - 6b\mathfrak{Re} (\dot{c}_0\ov{\dot{c}_1}) - 8b\mathfrak{Re} (\dot{c}_0\dot{c}_1) +2\sqrt{2} \mathfrak{Re} (\dot{c}_0\dot{c}_2) - 2\sqrt{2}b\mathfrak{Re} (\dot{c}_1\ov{\dot{c}_2}),
 \end{multline*}
  \begin{multline*}
 \mathcal{Q}_2=
 \sum_{n=3}^{4} \big[ (n-4+\frac{8}{2^n})b^2+n-3+\frac{4(n+1)}{2^n}\big]|\dot c_n|^2+  2b \mathfrak{Re} \sum_{n=3}^{4}(1-\frac4{2^n})\sqrt{n+1} \dot{c_n} \ov{\dot{c}_{n+1}}   \\
 = \sum_{n=3}^{4} \big[  n-3+\frac{4(n+1)}{2^n}\big]|\dot c_n|^2+b^2\sum_{n=3}^{4} \big[  n-3+\frac{4}{2^n}\big]|\dot c_{n+1}|^2+  2b \mathfrak{Re} \sum_{n=3}^{4}(1-\frac4{2^n})\sqrt{n+1} \dot{c_n} \ov{\dot{c}_{n+1}}  
  \end{multline*}
 and
  \begin{equation*}
 \mathcal{Q}_3
 = \sum_{n=5}^{+\infty} \big[  n-3+\frac{4(n+1)}{2^n}\big]|\dot c_n|^2+b^2\sum_{n=5}^{+\infty} \big[  n-3+\frac{4}{2^n}\big]|\dot c_{n+1}|^2+  2b \mathfrak{Re} \sum_{n=5}^{+\infty}(1-\frac4{2^n})\sqrt{n+1} \dot{c_n} \ov{\dot{c}_{n+1}}  
  \end{equation*}
(one can notice that the interaction $\mathfrak{Re} (\dot{c}_2\ov{\dot{c}_3})$ vanishes in \eqref{517}).\medskip

Let us now study the sign of \eqref{517}.\medskip

\underline{The quadratic form $ \mathcal{Q}_3$ is positive definite} : 
For $n\geq 5$ one has the equality
$$(1-\frac4{2^n})^2(n+1) < \big( n-3+\frac{4(n+1)}{2^n}\big)\big( n-3+\frac{4}{2^n}\big),$$
then one get $ \mathcal{Q}_3 > 0$.\medskip

\underline{The quadratic form $ \mathcal{Q}_2+\mathcal{Q}_3$ is positive definite} : Set ${c_j=x_j+iy_j}$, then 
$$\mathcal{Q}_2= 2(x_3+\frac{b}2x_4)^2+2(y_3+\frac{b}2y_4)^2+(\frac32 x_4+\frac{\sqrt{5}b}2 x_5)^2+(\frac32 y_4+\frac{\sqrt{5}b}2 y_5)^2,$$
and the claim follows. \medskip

\underline{Under the constraints \eqref{cond1} and  \eqref{cond2}, the quadratic form $ \mathcal{Q}_1$ is non-negative} : 
Set ${c_j=x_j+iy_j}$. Then  \eqref{cond1} and  \eqref{cond2} imply that $x_1=bx_0$ and $x_2=-\sqrt{2}x_0$. Therefore 
\begin{eqnarray*}
\mathcal{Q}_1&=&(b^2+1)^2 x^2_0+ y^2_0+b^2y^2_1+2y^2_2+2b y_0y_1-2\sqrt{2}y_0y_2-2\sqrt{2}by_1y_2\\[3pt]
 &=&(b^2+1)^2 x^2_0+ (y_0+by_1-\sqrt{2}y_2)^2.
\end{eqnarray*}
The matrix of this quadratic form has two positive eigenvalues \big($(b^2+1)^2$ and $(b^2+3)$\big), and  the eigenvalue 0 has multiplicity 2, which corresponds to the symmetries $T_\gamma$ and $L_{\phi}$.
\end{proof}

\subsection{Stability of stationary waves with finite mass and a finite number of zeros}

\begin{thm} \begin{itemize}
\item[(i)] The stationary wave $\phi_0^\alpha$, for $\alpha \in \mathbb{C}$, is orbitally stable in $L^2$ for the symmetries of the equation. More precisely, there exists $C>0, \delta _0>0$ such that, if  $\| u_0 - \phi_0^\alpha \|_{L^2(\C)}=\delta \leq \delta _0$, then   the associated solution $u$ of~\eqref{LLL} satisfies
$$
\sup_{t\in \R}\inf_{\theta \in \mathbb{T}, \,\beta \in \mathbb{C}} \left\| u(t) - e^{i\theta} \phi_0^\beta \right\|_{L^2(\C)} \leq C\sqrt{\delta }.
$$
\item[(ii)] The stationary waves $\phi_0^\alpha$ and $\phi_1^\alpha$ are orbitally stable in $L^{2,1}$ for the phase rotation symmetry. More precisely, there exists $C>0, \delta _0>0$ such that, if $j=0$ or $1$, $\| u_0 - \phi_j\|_{L^{2,1}(\C)}=\delta \leq \delta _0$, then   the associated solution $u$ of~\eqref{LLL} satisfies
$$
\sup_{t\in \R}\inf_{\theta \in \mathbb{T}} \| u(t) - e^{i\theta} \phi_j \|_{L^{2,1}(\C)} \leq C\sqrt{\delta }\ .
$$
\item[(iii)] For all $b\geq 0$, the stationary waves $\psi_b$   are orbitally stable in $L^{2,1}$ for the phase rotation and the space rotation. More precisely, there exists $C>0, \delta _0>0$ such that  $\| u_0 - \psi_b\|_{L^{2,1}(\C)}=\delta \leq \delta _0$, then   the associated solution $u$ of~\eqref{LLL} satisfies
$$
\sup_{t\in \R}\inf_{\theta \in \mathbb{T}, s\in \R} \| u(t) - e^{i\theta} L_s\psi_b\|_{L^{2,1}(\C)} \leq C\sqrt{\delta }\ .
$$
\item[(iv)] More generally,  consider $v_0(z)=(\lambda_0 z + \mu_0) e^{\alpha_0 z - \frac{1}{2} |z|^2}$. Then  there exists $C>0, \delta _0>0$ such that  $\| u_0 -v_0 \|_{L^{2,1}(\C)}=\delta \leq \delta _0$, then   the associated solution $u$ of~\eqref{LLL} satisfies
$$
\sup_{t\in \R}\inf_{\theta \in \mathbb{T}, s\in \R, \alpha \in \C} \| u(t) - e^{i\theta} L_sR_{\alpha}\psi_b\|_{L^{2,1}(\C)} \leq C\sqrt{\delta }\ .
$$
\item[(v)] The stationary waves $\phi_n^\alpha$, $n \geq 2$, are not orbitally stable.
\end{itemize}
\end{thm}

Numerical evidence for the above stability results can be found in \cite{BBCE}.

\begin{proof} The proofs of $(i)$, $(ii)$, $(iii)$, and $(iv)$ are variational. Indeed, assertion $(ii)$ follows from Proposition~\ref{proplocalmin}, assertion~$(iii)$ from Proposition \ref{proplocalmin2}, and assertion~$(iv)$ from Proposition \ref{prop43}. As for property $(i)$, it is a consequence of the following observation : the Hessian $\mathcal L$ of $\frac 12M-2\pi \mathcal H$ has a kernel 
spanned by $i\phi_0, \phi_1, i\phi_1$, and it satisfies
$$\mathcal L(\phi_0)=-2\phi_0\ ,\ \mathcal L(\phi _n)=(1-2^{1-n})\phi_n\ ,\ \mathcal L(i\phi_n)=(1-2^{1-n})i\phi_n, n\geq 2\ .$$
This implies the following bound, from which $(i)$ follows easily. 
\begin{lem}
If $\delta _0>0$ is small enough and 
$$\delta (u):=|M(u)-M(\phi_0)|+|\mathcal H(u)-\mathcal H(\phi_0)|\leq \delta _0\ ,$$
then
$$\inf_{(\theta ,\beta )\in \mathbb {T}\times \mathbb {C}}\Vert u-{\rm e}^{i\theta}R_\beta \phi_0\Vert_{L^2}^2\leq \delta (u)\ .$$
\end{lem}
\begin{proof} By contradiction, combining Corollary \ref{coromin}, modulation by the group $\mathbb{T}\times \mathbb{C}$, and the following coercivity estimate,
$$\forall h\in \mathcal E, C^{-1}\| h\| _{L^2}^2\leq (\mathcal Lh,h)+C(h,\phi_0)^2+(h,i\phi_0)^2+(h,\phi_1)^2+(h,i\phi_1)^2\ ,$$
where $(f,g)$ denotes the real part of the inner product of $f,g \in L^2$. Details are left to the reader.
\end{proof}
Finally, the proof of $(v)$ is mostly contained in~\cite[Section 8.2]{GHT1}, but we include it here for the sake of completeness. Up to the symmetries of the equation, it suffices to consider the stationary wave 
$$
\phi_n e^{-i\omega_n t}, \quad \mbox{with} \quad \omega_n = \frac{(2n)!}{\pi (n!)^2 2^{2n+1}}.
$$
Switching to the variable, $d_k = e^{i \omega_n t} c_k$, the linearized equation reads
$$
\left\{
\begin{array}{ll}
i \partial_t d_n = \omega_n d_n + \omega_n \overline{d_n} &  \\
i \partial_t d_k = (\alpha_{n,k} - \omega_n) d_k + \beta_{n,k} \overline{d_{2n-k}} & \mbox{if $k \leq 2n$} \\
i \partial_t d_k =  (\alpha_{n,k} - \omega_n) d_k & \mbox{if $k \geq 2n + 1$},
\end{array}
\right.
$$
where $\alpha_{n,k} = \frac{(n+k)!}{\pi n! k ! 2^{n+k+1}}$ and $\beta_{n,k} = \frac{(2n)!}{\pi n! \sqrt{k ! (2n-k)!} 2^{2n+1}}$. The equation for $d_n$ gives linear growth at most (corresponding to the phase invariance), while the equation for $d_k$, with $k \geq 2n+1$ is obviously stable. Turning to the modes $\leq 2n$, $k$ and $2n-k$ are coupled. Setting $d_k=x$, it satisfies the equation
$$
\ddot{x} + i \big(\alpha_{n,k} - \alpha_{n,2n-k}\big)\dot x  -  \big(\beta_{n,k}^2 - (\alpha_{n,k} - \omega_n)(\alpha_{n,2n-k} - \omega_n)\big) x = 0.
$$
This equation has unstable (exponentially growing) modes if and only if the discriminant
$$
\Delta_{n,k} = 4\beta_{n,k}^2 - (\alpha_{n,k} + \alpha_{n,2n-k}- 2\omega_n)^2 > 0.
$$
A computation shows that $\Delta_{n,n-2} > 0$, giving the desired (linear) instability. The next step is classical: linear instability implies nonlinear instability. A proof of this can be found e.g. in   \cite[Section 6]{GSS2}.
\end{proof}

\appendix

 \section{Some explicit  $M$-stationary waves}\label{explicitstatwave}
 
 We start with stationary waves having simple zeros at $\gamma {\mathbb Z}$ for some complex number $\gamma \ne 0$.
 \begin{prop} For $\alpha \in \C$, $\alpha \neq 0$ the function
$$\chi_{\alpha}(z)=\frac{e^{\alpha z}-e^{-\alpha z}}{\sqrt{2\pi(e^{ |\alpha|^2}-e^{- |\alpha|^2})}}\, e^{-\frac{|z|^2}{2}}=\frac{\sinh(\alpha z)}{\sqrt{\pi \sinh(|\alpha |^2)}}\, e^{-\frac{|z|^2}{2}}\ ,$$
is an $M$-stationary wave in $\mathcal{E}$ which has an infinite number of zeros. It satisfies
$$
\mathcal{H}(\chi_\alpha) = \frac{1}{16\pi}, \;\;\;\;
M(\chi_{\alpha}) =1, \;\;\;\; P(\chi_{\alpha}) = |\alpha|^2 \frac{e^{|\alpha|^2}+e^{-|\alpha|^2}}{e^{|\alpha|^2}-e^{-|\alpha|^2}}, \;\;\;\; Q(\chi_{\alpha}) = 0.
$$
The corresponding solution to \eqref{LLL} is $\chi_\alpha e^{-i\lambda t}$ with $\lambda= \frac{1}{4\pi}.$
\end{prop}

\begin{proof} Set
 $$\chi_{\alpha}(z)=A(e^{\alpha z}-e^{-\alpha z}) e^{-\frac{|z|^2}{2}}=\sqrt{\pi}e^{\frac{|\alpha |^2}{2}}A(\phi_0^{\alpha}-\phi_0^{-\alpha})(z),$$
 where $A>0$ is such that $M(\chi_{\alpha}) =1$. Then,  by \eqref{defpi} and \eqref{gaussian}, 
 \begin{multline*}
 \Pi\big[|\chi_\alpha|^2 \chi_\alpha\big](z)\\
   \begin{aligned}
& =A^3  \Pi\Big[e^{-\frac{3|z|^2}{2}}(e^{2 \alpha z+\ov \alpha \ov{z}}-e^{2 \alpha z-\ov \alpha \ov{z}}+e^{-2 \alpha z+\ov \alpha \ov{z}}-e^{-2 \alpha z-\ov \alpha \ov{z}}-2e^{ \ov \alpha \ov z}+2e^{-\ov \alpha \ov{z}})\Big]\\
 &=\frac{A^3}{\pi}e^{-\frac{|z|^2}{2}} \int_{\C} \Big[e^{-{2|w|^2}+\ov{w}z}(e^{2 \alpha w+\ov \alpha \ov{w}}-e^{2 \alpha w-\ov \alpha \ov{w}}+e^{-2 \alpha w+\ov \alpha \ov{w}}-e^{-2 \alpha w-\ov \alpha \ov{w}}-2e^{ \ov \alpha \ov w}+2e^{-\ov \alpha \ov{w}})\Big]dL(w)\\
 &=\frac{A^3}{2} (e^{\alpha z+|\alpha |^2}-e^{\alpha z-|\alpha |^2}+e^{-\alpha z-|\alpha |^2}-e^{-\alpha z+|\alpha |^2})e^{-\frac{|z|^2}{2}}\\
  &=\frac{A^3}{2} (e^{|\alpha |^2}-e^{-|\alpha |^2})(e^{\alpha z}-e^{-\alpha z})e^{-\frac{|z|^2}{2}} ,
 \end{aligned}
  \end{multline*}
  which shows that $\chi_{\alpha}$ is a  $M$-stationary wave with $\lambda =\frac{1}{2}A^2(e^{|\alpha |^2}-e^{-|\alpha |^2})$ and from the previous lines we have $\mathcal{H}(\chi_\alpha) = \frac{1}{4} \lambda$.  Set 
 $$v_{\alpha}=\phi_0^{\alpha}-\phi_0^{-\alpha}.$$
By \eqref{gaussian} we have
 \begin{eqnarray*}
M(v_\alpha)&=&\int |\phi_0^{\alpha}|^2+\int |\phi_0^{-\alpha}|^2-2 \Re \int \phi_0^{\alpha}  \ov{\phi_0^{-\alpha}} \\
&=& 2-\frac2{\pi}\Re \int  e^{-{|z|^2+\alpha z -\ov \alpha \ov{z}}-|\alpha |^2}=2(1-e^{-2 |\alpha |^2}),\\
\end{eqnarray*}
which gives the values $A=\big[2\pi(e^{ |\alpha|^2}-e^{- |\alpha|^2})\big]^{-1/2}$, $\lambda=1/(4\pi)$ and $\mathcal{H}(\chi_\alpha) =1/(16\pi)$. Next, 
\begin{align*}
\int |z|^2 |v_\alpha|^2&=\int  |z|^2|\phi_0^{\alpha}|^2+\int  |z|^2 |\phi_0^{-\alpha}|^2-2 \Re \int  |z|^2\phi_0^{\alpha}  \ov{\phi_0^{-\alpha}} \\
&= 2|\alpha|^2+2-\frac{2}\pi\Re \int  |z|^2 e^{-{|z|^2+\alpha z -\ov \alpha \ov{z}}-|\alpha |^2}\\
& = \left.  2|\alpha |^2+2-2e^{-|\alpha |^2}\partial_A \partial_B   e^{AB} \right|_{\substack{A=\alpha \\ B =-\alpha  }}\\
 & =2|\alpha |^2+2-2(1-|\alpha |^2)e^{-2|\alpha |^2},
\end{align*}
thus  $P(v_\alpha)=2|\alpha |^2(1+e^{-2 |\alpha |^2})$.

Finally, $Q(\chi_\alpha )=0$ follows from $|\chi_\alpha (-z)|=|\chi_\alpha (z)|$.
\end{proof}
Our second example provides stationary waves having zeros located on $\gamma {\mathbb Z} \cup \frac{i\pi }{k\ov \gamma}{\mathbb Z}$ for some $\gamma \ne 0$ and for some integer $k\ne 0$. 

\begin{prop} For $k \in \Z$ and $\alpha \in \C$ with $k,\alpha \neq 0$,  the function
$$v_{k}(z)=\frac{ \sinh(\alpha z) \sin\left (\frac{k\pi z}{\ov \alpha }\right )}{\sqrt{\pi \sinh(|\alpha |^2)\sinh\left (\frac{k^2\pi^2}{|\alpha |^2}\right )}}\, e^{-\frac{|z|^2}{2}},$$
  is an $M$-stationary wave in $\mathcal{E}$ which has an infinite number of zeros. It satisfies
$$
\mathcal{H}(v_k) = \frac{1}{32\pi}, \;\;\;\;
M(v_k) =1, \;\;\;\; Q(v_k) = 0.
$$
$$
P(v_k)=
\frac{(|\alpha|^2+\frac{\pi^2 k^2}{|\alpha|^2})(e^{|\alpha|^2+ \frac{\pi^2 k^2}{|\alpha|^2} }-e^{-|\alpha|^2- \frac{\pi^2 k^2}{|\alpha|^2} })+(|\alpha|^2-\frac{\pi^2 k^2}{|\alpha|^2})(e^{-|\alpha|^2+ \frac{\pi^2 k^2}{|\alpha|^2} }-e^{|\alpha|^2- \frac{\pi^2 k^2}{|\alpha|^2} }) }{(e^{ |\alpha|^2}-e^{- |\alpha|^2})(e^{ \frac{\pi^2 k^2}{|\alpha|^2}   }-e^{-  \frac{\pi^2 k^2}{|\alpha|^2} })}.
$$
The corresponding solution to \eqref{LLL} is $v_k e^{-i\lambda t}$ with $\lambda= \frac{1}{8\pi}.$
\end{prop}
\begin{proof}
 Set $\theta_{k}(z):= (e^{ \alpha z}-e^{-\alpha z}) e^{  \frac{i\pi k}{\ov \alpha} z}e^{-\frac{|z|^2}{2}}$.
First of all we show, using \eqref{defpi} and \eqref{gaussian}, that, for all $k_1,k_2,k_3 \in \Z$ such that $k_1,k_2$ have the same parity, 
$$\Pi\big[\theta_{k_1}\theta_{k_2}\ov{\theta_{k_3}}\big]=\frac12 (e^{|\alpha|^2}-e^{-{|\alpha|^2}})(-1)^{k_3+\frac{k_1+k_2}2} e^{\frac{\pi^2(k_1+k_2)k_3}{2|\alpha|^2}}\theta_{\frac12(k_1+k_2)}.$$
Then write $$v_k=-iA (e^{  \alpha z}-e^{- \alpha z}) (e^{  \frac{i\pi k}{\ov{\alpha}} z}-e^{-  \frac{i\pi k}{\ov{\alpha}} z})e^{-\frac{|z|^2}{2}}=-iA(\theta _k-\theta _{-k})\ ,$$ with $A>0$ such that $M(v_k)=1$. We obtain, from the above identity,
 \begin{eqnarray*}
 \Pi\big[|v_k|^2 v_k\big]&=&   -iA^3\Pi\big[\theta_k^2 \ov{\theta_k} -\theta_k^2 \ov{\theta_{-k}}+\theta_{-k}^2 \ov{\theta_{k}}-\theta_{-k}^2 \ov{\theta_{-k}}-2\theta_k \theta_{-k}\ov{\theta_k}+2\theta_{-k} \theta_{k}\ov{\theta_{-k}}\big]\\
 &=&\frac {A^2}2(e^{ |\alpha|^2}-e^{- |\alpha|^2})(e^{ \frac{\pi^2 k^2}{|\alpha|^2}   }-e^{-  \frac{\pi^2 k^2}{|\alpha|^2} })v_k.
  \end{eqnarray*}
Therefore $\lambda=\frac {A^2}2(e^{ |\alpha|^2}-e^{- |\alpha|^2})(e^{ \frac{\pi^2 k^2}{|\alpha|^2}   }-e^{-  \frac{\pi^2 k^2}{|\alpha|^2} })$ and $\mathcal{H}(v_k) = \frac{\lambda}{4}$. Then we compute
$$M(v_k)=4\pi A^2(e^{ |\alpha|^2}-e^{- |\alpha|^2})(e^{ \frac{\pi^2 k^2}{|\alpha|^2}   }-e^{-  \frac{\pi^2 k^2}{|\alpha|^2} }) =1,$$
which provides the value of $A$. Finally, with a repeated use of the formula
$$ \frac1{\pi}\int_\C (|w|^2-1)e^{-|w|^2+Aw+B \ov{w}}dL(w)=AB e^{AB}$$
we get
\begin{multline*}
P(v_k)=\\
=4 \pi A^2 \left [\left (|\alpha|^2+\frac{\pi^2 k^2}{|\alpha|^2}\right )\left (e^{|\alpha|^2+ \frac{\pi^2 k^2}{|\alpha|^2} }-e^{-|\alpha|^2- \frac{\pi^2 k^2}{|\alpha|^2} }\right )+\left (|\alpha|^2-\frac{\pi^2 k^2}{|\alpha|^2}\right )\left (e^{-|\alpha|^2+ \frac{\pi^2 k^2}{|\alpha|^2} }-e^{|\alpha|^2- \frac{\pi^2 k^2}{|\alpha|^2} }\right ) \right ].
\end{multline*}
\end{proof}

 \section{The dictionary}\label{dictionary}

For $f \in \mathscr{S}'(\R)$, we define the Bargmann transform $B$ by
$$(Bf)(z)=\frac1{\pi^{3/4}} e^{\frac{z^2}2} \int_{\R} e^{-\frac{(\sqrt{2}z-y)^2}2}f(y)dy,\quad z \in \C.$$
Then 
$$(B^{\star}u)(y)=\frac1{\pi^{3/4}} \int_{\C} e^{\frac{\ov{w}^2}2}  e^{-\frac{(\sqrt{2}\ov{w}-y)^2}2} e^{-\frac{|w|^2}2}u(w)dL(w),\quad y \in \R,$$
and a direct computation gives $B B^{\star}=e^{|z|^2/2}\Pi$ (see \cite{ABN} for more details on the Bargmann transform.) \medskip

In the following tabular, for each stationary wave $u$, we list the corresponding  coordinates $(c_k)$ such that $u=\sum_{k\geq 0} c_k \phi_k$, and  $f=B^{\star}u$. \medskip\medskip

\begin{tabular}{| c | c | c  |}
 \hline
 $u$ &  $c_k$ & $f$  \\[6pt]
 \hline
  \rule[-0.9cm]{0cm}{1.6cm}
$\displaystyle \varphi_0(z) = \frac{1}{\sqrt \pi} e^{-\frac{|z|^2}{2}}$ & $\displaystyle \delta_{0,k}$ & $\displaystyle \frac{1}{\pi^{1/4}} \displaystyle e^{-\frac{y^2}{2}}$  \\ 
 \hline
   \rule[-0.9cm]{0cm}{1.6cm}
$\displaystyle \varphi_0^\alpha(z) =  \frac{1}{\sqrt \pi} e^{-\frac{|z|^2}{2}-\frac{|\alpha|^2}{2}+\alpha z}$ & $\displaystyle \frac{\alpha^k}{\sqrt{k !}} e^{- \frac{|\alpha|^2}{2}}$ & $\displaystyle \frac{1}{\pi^{1/4}} e^{i \alpha_I (\sqrt{2} y - \alpha_R) - (\frac{y}{\sqrt 2} - \alpha_R)^2}$ \\[4pt]
 \hline
 \rule[-0.9cm]{0cm}{1.6cm}
$\displaystyle \varphi_n(z) = \frac{1}{\sqrt{ \pi n!}} z^n e^{-\frac{|z|^2}{2}}$ & $\displaystyle \delta_{n,k}$ & $\displaystyle \frac{1}{\pi^{1/4}2^{n/2} \sqrt{n!}} H_n(y)$ \\[6pt]
 \hline
 \rule[-0.9cm]{0cm}{1.6cm}
$\displaystyle e^{-\frac{|z|^2}{2} + \frac{z^2}{2}}$ & $\displaystyle \frac{\sqrt{\pi k!}}{2^{k/2} (k/2)!} \mathbf{1}_{k\; \mbox{even}}$ & $\displaystyle \frac{\pi^{1/4}}{\sqrt{2}}$ \\[6pt]
 \hline  
 \rule[-0.9cm]{0cm}{1.6cm}
$\displaystyle z e^{-\frac{|z|^2}{2} + \frac{z^2}{2}}$ & \;$\displaystyle \frac{\sqrt{\pi k!}}{2^{(k-1)/2} ((k-1)/2)!} \mathbf{1}_{k\; \mbox{odd}}$ \; & $\displaystyle \frac{\pi^{1/4}y}{2\sqrt 2}$  \\[6pt]
 \hline
\end{tabular}
\medskip

where
$$
\alpha = \alpha_R + i \alpha_I \qquad \mbox{and} \qquad H_n(y) := (-1)^n e^{\frac{y^2}{2}} (\partial_y)^n e^{-y^2}.
$$

\section{Sobolev spaces}  

Define the harmonic Sobolev spaces for $s \in \mathbb{R}$,   by 
\begin{equation*} 
\HH^{s}(\C) = \big\{ u\in \mathscr{S}'(\C),\; {H}^{s/2}u\in L^2(\C)\big\}.
\end{equation*}

This is a weighted Sobolev norm. In the Bargmann-Fock space, this norm simply corresponds to the weighted $L^{2,s}$-norm. In other words, regularity exactly corresponds to decay in the space variable. 

Precisely, setting $\<z\>=(1+|z|^2)^{1/2}$, we have the following result.
\begin{lem}\label{lemEq}
Let $s\in \R$. There exists $C>0$ such that for all  $u\in \tilde{\mathcal{E}}\cap \HH^s(\C)$
\begin{equation*}
\frac1C\|\<z\>^s u\|_{L^2(\C)} \leq  \|u\|_{\HH^s(\C)} \leq C\|\<z\>^s u\|_{L^2(\C)}.
\end{equation*}
\end{lem}
 \medskip

\begin{proof}
Write $u=\sum_{n\geq 0}c_n \phi_n$. On the one hand, we have  $H^su=\sum_{n\geq 0}2^s(n+1)^sc_n \phi_n$, therefore
\begin{equation}\label{soboL}
 \|u\|^2_{\HH ^s(\C)}=\int_{\C} \ov{u} H^su\,dL(z)=\sum_{n\geq 0}2^s(n+1)^s|c_n|^2.
\end{equation}
On the other hand, 
\begin{eqnarray*}
\|\<z\>^s u\|^2_{L^2(\C)}&=& \int_{\C} \<z\>^{2s} |u(z)|^2 dL(z)\\
&=& \sum_{n,m\geq 0} \int_{\C}\<z\>^{2s} c_n \ov{c_m} \phi_n(z) \ov{\phi_m(z) } dL(z)\\
&=&\frac1\pi \sum_{n,m\geq 0} \int_{\C}\<z\>^{2s} \frac{c_n \ov{c_m}}{\sqrt{n! m!}} z^n\ov{z}^m \e^{-|z|^2}  dL(z).
\end{eqnarray*}
Now, we make the polar change of variables $z=r\e^{i\theta}$ and use that $\int_0^{2\pi} \e^{i(n-m)\theta}d\theta =2\pi \delta_{n,m}$,
\begin{equation*}
\|\<z\>^s u\|^2_{L^2(\C)}= 2  \sum_{n\geq 0}  \frac{|c_n|^2}{n!}\int_{0}^{+\infty}\<r\>^{2s} r^{2n+1}  \e^{-r^2}  dr.
\end{equation*}
With the change of variables $t=r^2$ we get
\begin{equation}\label{soboK}
\|\<z\>^s u\|^2_{L^2(\C)}=   \sum_{n\geq 0}  \frac{|c_n|^2}{n!}\int_{0}^{+\infty}(1+t)^{s} t^{n}  \e^{-t}  dt.
\end{equation}
Finally, we use the Stirling formula twice ($n\geq 1$)
\begin{equation*}
\frac1c \big(\frac{n}{\e}\big)^{n}\sqrt{n} \leq n! \leq c \big(\frac{n}{\e}\big)^{n}\sqrt{n}, \qquad \frac1cn^{n+s}\e^{-n}\sqrt{n} \leq \int_{0}^{+\infty}(1+t)^{s} t^{n}  \e^{-t}  dt \leq cn^{n+s}\e^{-n}\sqrt{n} ,
\end{equation*}
and conclude with \eqref{soboK} that 
\begin{equation*} 
\frac1C \sum_{n\geq 0}(n+1)^s|c_n|^2 \leq \|\<z\>^s u\|^2_{L^2(\C)} \leq C \sum_{n\geq 0} (n+1)^s|c_n|^2,
\end{equation*}
which completes the proof thanks to \eqref{soboL}.
\end{proof}

\end{document}